\documentclass[12pt]{amsart}
\usepackage{fullpage}
 
\allowdisplaybreaks 
\usepackage{todonotes}
\usepackage{hyperref}
\usepackage{amsmath,amsfonts,amssymb}
\usepackage{mathrsfs}
\usepackage{verbatim}
\usepackage{amsthm}
\usepackage{mathrsfs} 
\usepackage{import}

\newtheorem{theorem}{Theorem}[section]
 \newtheorem{corollary}[theorem]{Corollary}
 \newtheorem{lemma}[theorem]{Lemma}
 \newtheorem{proposition}[theorem]{Proposition}
 \theoremstyle{definition}
 \newtheorem{definition}[theorem]{Definition}
 
 \theoremstyle{remark}
 \newtheorem{remark}[theorem]{Remark}
  \newtheorem{ex}[theorem]{Example}
 \numberwithin{equation}{section}

\def \bG {\mathbb G}
\def \bH {\mathbb H}

\def \bN {\mathbb N}

\def \bR {\mathbb R}

\def \bX {\mathbb X}
\def \bY {\mathbb Y}

\def \cD {\mathcal D}

\def \cV {\mathcal V}

\def \fg {\mathfrak g}
\def \fh {\mathfrak h}

\def \id {\text{\rm I}}

\def \eps {\varepsilon}
\def \vol {{\rm vol}}
\def \pr  {{\rm pr}}

\def \IM  {{\rm Im}\, }

\def \Mat {{\rm Mat}}

\def \gr {{\rm gr}}

\def \mod {{\rm mod}}
\def \diag {{\rm diag}}

\def \dalg {\tilde d}

\begin{document}
\author[V. Fischer \& F. Tripaldi]{V\'eronique Fischer and Francesca Tripaldi}

\address[V. Fischer]%
{University of Bath, Department of Mathematical Sciences, Bath, BA2 7AY, UK} 
\email{v.c.m.fischer@bath.ac.uk}

\address[F. Tripaldi]%
{Centro De Giorgi,
SNS,
Piazza dei Cavalieri 3,  56 126 Pisa, Italy.} 
\email{francesca.tripaldi@sns.it}

\title{Subcomplexes on filtered Riemannian manifolds}

\medskip



\maketitle

\begin{abstract}
In this paper, we present a general construction to extract subcomplexes from two distinct complexes on filtered Riemannian manifolds.
The first subcomplex computes the de Rham cohomology of the underlying manifold. On regular subRiemannian manifold equipped with a compatible Riemannian metric, it aligns locally with the so-called Rumin complex.
The second complex instead generalises the Chevalley-Eilenberg complex computing Lie algebra cohomology of a nilpotent Lie group. Our approach offers key insights on the role of the Riemannian metric when extracting subcomplexes, opening up potential new applications in more general geometric settings, such as singular subRiemannian manifolds. 
\end{abstract}

\makeatletter
\renewcommand\l@subsection{\@tocline{2}{0pt}{3pc}{5pc}{}}
\makeatother

\tableofcontents

\section{Introduction}

\subsection{Motivation}
\label{subsec_motivation}
The de Rham complex and its cohomology have inspired many mathematical ideas in the twentieth century, including Hodge theory and the Atiyah-Singer index theorem on Riemannian manifolds. More recently, an increasing number of techniques have been developed in order to extract subcomplexes with the aim of extending such classical results to more general geometric settings \cite{CaseChanilloYang,tripaldi2020rumin,DaveHaller22,case2022,CapHu,LerarioTripaldi}, especially in subRiemannian geometry.

In \cite{F+T1}, we proposed two constructions on homogeneous nilpotent Lie groups, one by adapting techniques previously developed in the context of parabolic geometry
\cite{CapSlovakSouvek,CalderbankDiemer,DaveHaller22} and the other influenced by the ideas in subRiemannian geometry and spectral sequences \cite{JulgSpectral,RuminPalermo}.
In the present paper, we push these ideas further to obtain subcomplexes from two different initial complexes and with two equivalent but different constructions. One one hand, we start from the usual de Rham complex $(\Omega^\bullet(M),d)$ to extract a subcomplex $(E_0^\bullet,D)$ computing the de Rham cohomology of the underlying manifold $M$, while on the other, we extract $(\tilde{E}_0^\bullet,\tilde{D})$ from the ``algebraic'' complex $(\Omega^\bullet(M),\tilde{d})$. Here $\tilde d$ denotes the algebraic part of the de Rham differential $d$, and so the complex $(\Omega^\bullet(M),\tilde{d})$ generalises the Chevalley-Eilenberg complex computing the Lie algebra cohomology of a given group to the manifold setting.

\subsection{Sketch of the constructions of the complexes}
\label{subsec_sketch}
Here, we explain our setting and briefly sketch our two equivalent constructions of each of the two complexes $(E_0^\bullet,D)$ and $(\tilde E_0^\bullet,\tilde D)$. The various claims will be proved in Section~\ref{sec_generalscheme}.

We consider  a filtered manifold $M$ and denote  by $\fg_x M$ its  osculating nilpotent Lie  algebra (also called nilpotentisation) above each point 
$x\in M$ (see Section \ref{subsec_def+setting}). 
We  then define 
the osculating Chevalley-Eilenberg differential $d_{\fg M}$ on the space $G\Omega^\bullet (M)$ of osculating forms over $M$, defined above each point $x\in M$ as the  Chevalley-Eilenberg differential of the osculating group $G_x M$  (see Section \ref{subsec_dfgM}).

We now further assume that $M$ is also equipped with a metric  on $TM$.
This allows us to construct the natural isomorphism
$\Phi:\Omega^\bullet (M)\to G\Omega^\bullet (M)$, 
providing an identification  between forms and osculating forms   (see Section \ref{subsec_filtration+metric}).
We are then able to define the map $d_0:=\Phi^{-1} \circ d_{\fg M}\circ \Phi$ acting on $\Omega^\bullet (M)$, 
which we call the base differential. 

\subsubsection{Constructing $D$ using $d_0^t$}
\label{subsubsec_constDd0t}
The  transpose $d_0^t$ of $d_0$  makes sense unambiguously and globally thanks to the metric on $TM$.
We consider the projection $\Pi_0$ 
onto 
$E_0:=  \ker d_0 \cap \ker d_0^t$ along $F_0 :=\IM d_0 + \IM d_0^t,$
and the projection $P$ 
onto 
$\ker d_0^t  \cap \ker d_0^t  d $
along $\IM d_0^t  + \IM d d_0^t$. 
The proper definitions are in fact given by 
  $\Pi_0$ being the orthogonal projection onto the kernel of the algebraic operator $\Box_0:=d_0 d_0^t +d_0^t d_0$, 
and  $P$ being the generalised kernel projection of the differential operator  
$\Box := d d_0^t +d_0^t d $.
Once we define the invertible differential operator
	$L:= P  \Pi_0 +(\id- P)(\id- \Pi_0)$,
  we have the decomposition 
 $$
 L^{-1}  d  L   =:  D+ C,
$$
yielding two complexes $(E_0^\bullet,D)$ and  $(F_0^\bullet,C)$ whose cohomologies are described as follows:

\smallskip

$\bullet$ 
$(F_0^\bullet, C)$ has trivial cohomology, since 
 $ C$ is conjugated to $d_0$ on $F_0$, and 
 $(F_0^\bullet,d_0)$ is acyclic.	

\smallskip

$\bullet$ The cohomology of $(E_0^\bullet, D)$ is linearly isomorphic to the de Rham cohomology of $M$ via the homotopically invertible chain map $ \Pi_0  L^{-1}$.

\subsubsection{Constructing $D$ using $d_0^{-1}$}
\label{subsubsec_constDd0inv}
Relying again on the metric on $TM$, we can define  the orthogonal projection $\pr_{\IM d_{\fg M}}$ onto $\IM d_{\fg M}$. We then consider the partial inverse $d_{\fg M}^{-1}:= (d_{\fg_x M})^{-1} \pr_{\IM d_{\fg_x M}}$ of $d_{\fg_x M}$ and thus the corresponding partial inverse 
$d_0^{-1}
	 := \Phi^{-1} \circ d_{\fg M}^{-1}\circ\Phi$ of $d_0$.
  Then the projection $P$ above may be constructed explicitly in terms of $d_0^{-1}$ as:
  $$
  P=(\id-b)^{-1} d_0^{-1}d+d(\id-b)^{-1} d_0^{-1}, \quad\mbox{where}\quad b:=d_0^{-1}d_0 - d_0^{-1} d.
  $$
The complex $D$ is then obtained as 
$D=\Pi_0 d (\id -P) \Pi_0$.

\subsubsection{Constructions of $\tilde D$}
We can modify the two constructions above by applying the same steps, but substituting the initial differential operator $d$ with the algebraic part $\tilde d$ of the de Rham differential instead (see Section \ref{subsec_ddalg}).
This yields two complexes 
$(\tilde E_0^\bullet,\tilde D)$ and  $(\tilde F_0^\bullet,\tilde C)$, 
with $(\tilde F_0^\bullet,\tilde C)$ acyclic
and $(\tilde E_0^\bullet,\tilde D)$ linearly isomorphic to $(\Omega^\bullet (M), \tilde d)$.

\subsection{Novelty and future works}
As already mentioned above,  constructions  related to the ones presented in this paper for the complex $D$ appeared before, in fact more than two decades ago in parabolic geometry 
\cite{CapSlovakSouvek,CalderbankDiemer,DaveHaller22}, 
and in relation to subRiemannian geometry and spectral sequences \cite{Rumin1990,JulgSpectral}. 
Already at that time, it was conjectured that these constructions should align in the resulting subcomplex in some sense (see Section 5.3 in \cite{RuminPalermo}, Section 5 in \cite{JulgKasparov}, Section 2.1 in \cite{Julg2019},
Section 5.3.3 in \cite{JulgSurvey}, and Section 1.3 in \cite{DaveHaller22}), 
but no proof was offered. 
An important difficulty incomparing the constructions is that the considerations in the subRiemannian settings were  either local or at the level of osculating objects (see Section 5.1 and Remark 5.2 in \cite{RuminPalermo}).

The constructions presented in this paper are set on a filtered Riemannian manifold $M$; they yield   complexes globally defined and  acting on $\Omega^\bullet (M)$ - not its osculating counterpart 
$G\Omega^\bullet (M)$. We obtain two equivalent constructions for the complex $(E_0^\bullet,D)$, which  in the subRiemannian world and together with its osculating couterparts are often referred to as the Rumin complex. We are then able to provide a clearer context and a proof to the conjecture explained above. We also show that these two equivalent constructions may be adapted to yield a different complex, that we have denoted by $(\tilde E_0^\bullet,\tilde D)$. 
  In Appendix \ref{sec_contact}, we give an explicit construction for the operators $D$ and $\tilde D$ on a 3D contact manifold. 

  \smallskip

  We have two  main examples in mind for our geometric setting of a filtered Riemannian manifold. We discuss them in turn.
\begin{ex}
\label{ex:sRM}
    A regular subRiemannian manifold  equipped with a compatible Riemannian metric (in the sense of Definition \ref{def_compatible}) is naturally a filtered Riemannian manifold.
\end{ex}

Within the setting of Example \ref{ex:sRM}, 
 we are able to show that the subcomplex $(E_0^\bullet,D)$ 
  coincides with what is now customarily referred to as the Rumin complex on regular $CC$-structures, first introduced
in \cite{Rumin1990,Rumin1999,RuminPalermo}.
 Part of the motivation behind the present work is to shed a new light on Rumin's construction on regular subRiemannian manifolds and on the nature of constructed objects (e.g. complexes acting on forms or on osculating forms). We also want  to address the potential impact that the choice of a Riemannian metric can have on the resulting subcomplex.  In particular, in this paper, we emphasise the nontrivial role that the extra hypothesis of compatibility between the Riemannian and subRiemannian structures plays in the construction, see Section \ref{subsubsec_compatibleRg}.
 Our hope is to highlight the behaviour of subcomplexes such as the Rumin complex, under changes of variables on regular subRiemannian manifolds \cite{franchitesi}  or  pullback by Pansu-differentiable maps on Carnot groups \cite{kleiner2021sobolev}. 

\begin{ex}
\label{ex_M=G}
    A nilpotent Lie group equipped with a left-invariant filtration is naturally a filtered Riemannian manifold.
\end{ex}

Within the setting of Example \ref{ex_M=G}, the complex $(\tilde E_0^\bullet,\tilde D)$ computes the Lie algebra cohomology of the given Lie group. 
A more thorough study of this setting, especially concerning the impact of the particular choice of a filtration on the resulting subcomplex $(E_0^\bullet,D)$,  will be presented in a forthcoming paper \cite{FTex532}.

\medskip

We selected the two specific settings in Examples \ref{ex:sRM} and \ref{ex_M=G} due to the exciting recent advances in their applications:
 \begin{itemize}
     \item large scale geometry
(especially  the open conjecture classifying  nilpotent Lie groups by quasi-isometries
\cite{PansuRumin,pansu2019averages,BaldiFranchiPansu22}),
\item analytic torsion \cite{RuminSeshadri,kitaoka2020analytic,Haller22}  and 
currents \cite{canarecci2021sub,Vittone,julia2023} for subRiemannian geometry,
and 
\item  $K$ and index theories  for subelliptic operators and $C^*$-algebras associated with subRiemannian structures
\cite{vanErp10,BaumvErp}.
 \end{itemize}
We believe that the techniques behind the subcomplexes $(E_0^\bullet,D)$ and $(\tilde E_0^\bullet,\tilde D)$  presented in Sections \ref{subsec_sketch}  and  \ref{sec_generalscheme} are general enough that their construction will be adaptable to other settings in the future.
Interesting generalisations would be to singular subRiemannian manifolds (e.g. Martinet distributions) and to the quasi-subRiemannian setting, for instance the Grushin plane.

\subsection{Organisation of the paper}
In order to present the hypotheses of our setting and to remove any ambiguity on the objects we consider, the paper starts with some foundational preliminaries on filtrations over vector spaces and on tangent bundles of manifolds (Section \ref{sec_prelim}).
When the manifold is filtered (i.e. when the tangent bundle is filtered and the commutator brackets of vector fields respect this filtration),
then we can define the osculating Chevalley-Eilenberg differential $d_{\fg M}$ and other osculating objects (see Section~\ref{sec_filtM+osculating}).
In Section~\ref{sec_generalscheme}, we present the general scheme for our construction, 
discussing the particular case of regular subRiemannian manifold at the end in Section \ref{subsec_regsubR}, and the particular case of 3D contact manifolds in Appendix \ref{sec_contact}.

\subsection{Acknowledgements}
Both authors acknowledge the support of The Leverhulme Trust through the Research Project Grant  2020-037, {\it Quantum limits for sub-elliptic operators}. We are also glad to thank the Centro di Ricerca Matematica Ennio De Giorgi and
the Scuola Normale Superiore for the hospitality and support in the early stages of the article, and Professor Giuseppe Tinaglia for hosting us at King's College London in October 2023.

\section{Preliminaries}
\label{sec_prelim}

The main objective of this section is to set some notation
for well-known notions regarding manifolds (Section \ref{sec_prelM})
and filtrations. 
For the latter, the filtrations are on  vector spaces and vector bundles (Sections \ref{sec_filtrationsVS} and \ref{sec_filtrationTM} respectively), which will come in handy when considering a manifold equipped with a metric  (Section \ref{subsec_filtration+metric}).

\subsection{Filtrations of vector spaces}
\label{sec_filtrationsVS}

Here, we briefly recall the construction of graded objects associated with a  filtration of vector spaces.
The notions presented below extend naturally to smooth vector bundles and their smooth sections, as well as to modules. 

In this paper, all the vector spaces are taken over the field of real numbers. 

\subsubsection{Filtered vector spaces}
A \textit{filtered vector space }$W$ is a vector space $W$ with a filtration by vector subspaces:
\begin{equation}
	\label{eq_incrF}
\{0\} = W_0 \subseteq W_1 \subseteq \ldots \subseteq W_s =W.
\end{equation}
In general, a  filtration need not be finite; however, in this paper,  
 all the filtrations considered will be finite as above.

A \textit{graded vector space} $G$ is a vector space $G$ equipped with a gradation by vector subspaces, i.e. a direct sum decomposition by vector subspaces $G_j\subset G$:
\begin{align}
\label{grading def}
    G=\oplus_{j=1}^sG_j\,.
\end{align}
If $G$ is graded, there is a natural \textit{filtration associated with its gradation} \eqref{grading def}, given by $\lbrace 0\rbrace=W_0\subseteq W_1\subseteq\cdots\subseteq W_s=G$, where  $W_j:=\oplus_{k\le j}G_k$, $j=1,\ldots,s$.

\subsubsection{The vector space $\gr ( W)$}
Given a filtered vector space $W$ with the filtration \eqref{eq_incrF}, its \emph{associated graded space} is the vector space
\begin{equation}
	\label{eq_gr_incrF}
	\gr ( W) := \oplus_{j=1}^s W_j/W_{j-1}\,.
\end{equation}

Many of the summands above may be trivial. To avoid dealing with zero quotient spaces, it is enough to consider an appropriate choice of indices 
 $w_0,w_1,\ldots ,w_{s_0}\in\mathbb N $ for which the summands in \eqref{eq_gr_incrF} are not trivial, or equivalently for which the inclusions in \eqref{eq_incrF} are strict:
\begin{align}\label{def strict filtration}
    \{0\} = W_{w_0} \subsetneq W_{w_1} \subsetneq \ldots \subsetneq W_{w_{s_0}} = W\ , 
\ \mbox{and}\  \gr \, (W) = \oplus_{j=1}^{s_0} W_{w_j}/W_{w_{j-1}}.
\end{align}

\subsubsection{The linear map $\gr  (T)$}
Consider a linear map $T: W \to W'$ between two vector spaces, each equipped with a filtration 
$\{0\} = W_0 \subseteq W_1 \subseteq \ldots \subseteq W_s =W$
and 
$\{0\} = W'_0 \subseteq W'_1 \subseteq \ldots \subseteq W'_{s'} =W'$. We may assume $s=s'$. 

The map $T$ is said to \emph{respect the filtrations} when 
$ T(W_j) \subseteq W'_j$ for every $j=0,\ldots, s$. 
In this case, one can define the linear map
$\gr  (T) : \gr\,( W) \to \gr \,( W')$
as
$$
\gr( T)\,(v \ \mod \ W_j)
:=( T v)  \ \mod \ W'_j\ \  ,
\ v \in W_{j+1}\ , \ j=0,\ldots, s-1\,.
$$

\subsubsection{The basis $\langle e\rangle$}
\label{subsubsec_anglee}
If $W$ is finite dimensional, then $\gr ( W)$ has  the same dimension as $W$.
Given $n:=\dim W=\dim \gr  (W)$ and $n_j = \dim W_{w_j}/W_{w_{j-1}}$, we say that a basis $e=(e_1,\ldots, e_n)$ of $W$ is \emph{adapted to the filtration} \eqref{def strict filtration}, when 
$(e_1,\ldots, e_{n_1})$ is a basis of $W_{w_{1}}$, 
$(e_{1},\ldots, e_{n_2+n_1})$ is a basis of $W_{w_2}$, and so on.
For such a basis, we have that, for each $j=0,\ldots,s_0-1$,
$$\langle e_i \rangle = e_i \ \mod \ W_{w_j} , \quad i=d_j+1,\ldots, d_{j+1} \ , \text{ with } d_j=n_0+n_1+\cdots+n_j\,,
$$
gives a basis $\langle e\rangle := (\langle e_1 \rangle, \ldots ,\langle e_n\rangle) $ of $\gr( W)$.

In particular, we get that for each $j=0,\ldots,s_0-1$, the $n_j$-tuple $(\langle e_{d_j+1}\rangle,\ldots,\langle e_{d_{j+1}}\rangle)$ is a basis of $W_{w_{j+1}}/W_{w_j}$. In this sense, the basis $\langle e\rangle$ is graded and is said to be the graded basis associated with $e$.

\subsubsection{Matrix representation of linear maps}
\label{subsubsec_Mrep}
	In this paper, we will often use matrix representations of certain maps. To fix some notation, given a morphism  $\varphi:\cV_1\to\cV_2$ between two finite dimensional vector spaces, once we fix $\beta_1$ and $\beta_2$ two bases of $\cV_1$ and $\cV_2$ respectively, we denote by 
$$
\Mat_{\beta_1}^{\beta_2} \,\varphi
$$
the matrix representing $\varphi$ in the bases $\beta_1$ and $\beta_2$.

Let us take $e$ and $e'$ two bases adapted to the filtrations \eqref{def strict filtration} of two finite dimensional filtered vector spaces $W$ and $W'$ respectively. Assuming $s_0=s'_0$, $n_j=\dim W_{w_j}/W_{w_{j-1}}$ and $n'_j=\dim W'_{w_j}/W'_{w_{j-1}}$, if a linear map $T\colon W\to W'$ respects the filtrations, the matrix representation of $T$ in the bases $e,e'$ is block-upper-triangular, i.e.
$$
M:=\Mat_e^{e'} \,T = 
\left( \begin{array}{ccccc}
 M_1 & * &*\\
 0& \ddots & * \\
 0& & M_{s} \\	
 \end{array}\right), 
$$
with $M_j \in \mathbb R^{n'_j}\times\mathbb R^{n_j}$, with $j=1,\ldots,s_0$.
Moreover, the matrix representing $\gr (T)$ in $\langle e\rangle, \langle e'\rangle$ is block-diagonal, i.e.
 $$\gr(M):=
\Mat_{\langle e\rangle}^{\langle e'\rangle} \,\gr(T) = 
\left( \begin{array}{ccccc}
 M_1 & 0 &0\\
 0& \ddots & 0 \\
 0& & M_{s} \\	
 \end{array}\right)\,.
$$

\subsubsection{Subfiltration}
We say that a vector subspace $\Tilde{W}\subseteq W$ admits a subfiltration when 
$\Tilde{W}_{w_j} := \Tilde{W} \cap W_{w_j}$, $j=1,\ldots, s_0$ yields a filtration of $\Tilde{W}$. 
In this case,  the injection map $\Tilde{W}\hookrightarrow W$ respects the filtrations and  $\gr\, (\Tilde{W})$ is a subspace of $\gr\,( W)$.


\subsubsection{Decreasing labelling and duality}
\label{subsubsec_duality}

One can also consider a filtered vector space $W$ with a decreasing filtration, that is a filtration with decreasing labelling:
\begin{align}\label{def decreasing filtr}
    W =V_0 \supseteq V_1 \supseteq \ldots \supseteq V_t =\{0\}\, .
\end{align}
 
It is not difficult to adapt all the constructions considered previously to a decreasing filtration.
Indeed, it suffices to change the labels by setting $W_j := V_{t-j}$, so that, for instance, the associated graded space becomes 
$\gr \, (V) := \oplus_{j=0}^{t-1} V_j/V_{j+1}$.
However, in this setting, the matrix representations will differ from those in Subsection \ref{subsubsec_Mrep}, as the matrix representing a linear map that respects such decreasing filtrations will be block-lower-triangular.

We observe that if $W$ has a decreasing filtration, then we obtain an increasing filtration of the dual space $W^\ast$ of $W$ by considering the filtration by the subspaces $V^\perp_j\subset W^\ast$, the annihilators of the subspaces $V_j\subset W$, that is
$$
\{0\} = V_0^\perp \subseteq V_1^\perp \subseteq \ldots \subseteq V_t^\perp = W^\ast \ ,\  V_j^\perp = \{\ell\in W^\ast\mid \ell |_{V_j}\equiv 0\} \,,
$$ 
also known as the dual filtration of \eqref{def decreasing filtr}.

Given two finite dimensional filtered vector spaces $W$ and $W'$, let us take $e$ and $e'$ two bases adapted to the decreasing filtrations $W=V_0\supseteq V_1\supseteq \cdots\supseteq V_t=\lbrace 0\rbrace$ and $W'=V'_0\supseteq V'_1\supseteq\cdots\supseteq V'_t=\lbrace 0\rbrace$. Then for any linear map $T\colon W\to W'$ that respects these decreasing filtrations, the dual map $$T^\ast\colon W'^\ast\to W^\ast\ \text{ where }\ T^\ast(\ell)=\ell\circ T\ , \ \forall\ \ell\in W'^\ast$$  respects the dual filtrations, and
\begin{align*}
    \Mat_{e'^\ast}^{e^\ast} T^\ast=\big(\Mat_e^{e'}T\big)^t\,,
\end{align*}
where $e^\ast$ and $e'^\ast$ denote the dual bases of $e$ and $e'$ respectively (one can easily show that if a basis $e$ is adapted to a decreasing filtration \eqref{def decreasing filtr}, then $e^\ast$ is adapted to its increasing dual filtration). In particular, this readily implies that the matrix representation of $T^\ast\colon W'^\ast\to W^\ast$ is block-upper-triangular, as pointed out in Subection \ref{subsubsec_Mrep}.

Moreover, the vector spaces $\gr (V^*)$ and $(\gr\, (V))^*$ are canonically isomorphic.

\subsubsection{Graded vector spaces}
\label{subsubsec_gradation}

As explained earlier, given a graded vector space $G$, its gradation \eqref{grading def} naturally produces a filtration through $\oplus_{k\leq j}G_{k}$, while a filtration \eqref{eq_incrF} of a vector space $W$ yields a graded space $\gr\, (W)$ \eqref{eq_gr_incrF}.

\begin{definition}
Let $G = \oplus_{j=1}^s G_j$ and 
	$G' = \oplus_{j'=1}^{s'} G'_j$ be two graded vector spaces. We may assume $s=s'$. 
	A  linear map $T:G\to G'$  respects the gradations when $T (G_j) \subset G'_j$ for any $j=1,\ldots,s$. 
\end{definition}
For instance, if $T$ is a linear map between two filtered vector spaces that respects the filtrations, then $\gr\,( T)$ respects the gradations.

Consider a graded vector space $G$ as above. 
Then its dual $G^*$ is also  graded  with the dual gradation $G^* = \oplus_{j=1}^s G_j^\perp$.
The $k^{th}$ exterior algebra is also graded via the (possibly infinite) gradation
$$
\wedge^k G = \oplus_{w=1}^\infty G^{k,w}
\, \text{, where }\
G^{ k,w}:=
\sum _{j_1 +\cdots+ j_k =w} G_{j_1} \wedge \cdots \wedge G_{j_k}\,.
$$
This leads to a (possibly infinite) gradation of the exterior algebra $\wedge^\bullet G$ which respects the wedge product:
$$
\wedge^\bullet G = \oplus_{w=0}^\infty G^{\bullet, w}, 
\qquad G^{\bullet, w} = \oplus_{k=0}^\infty G^{k,w},
$$
with $\wedge^0 G = G^{0,0}=\mathbb R$, and $G^{0,w}=\{0\}$ for $w>0$.
If $G$ is finite dimensional, the gradations of $\wedge^k G$ and $\wedge^\bullet G$ are finite. 

In the graded setting, there is a natural notion of weight: an element in $G_j$ has weight $j$, and more generally, an element in $G^{\bullet, w}$ has weight $w$.

\subsubsection{Filtration on the exterior algebra.}
A filtration \eqref{eq_incrF} on a vector space $W$ also induces a decreasing filtration on its $k^{th}$ exterior algebra:
\begin{equation}
	\label{eq_filtrationwedgekW}
	\wedge^k W=W^{ k, \geq 0}
\supseteq W^{ k, \geq 1}
\supseteq \ldots \supseteq
 W^{ k, \geq k\cdot s+1} =\{0\},
\end{equation}
where
$$
W^{ k, \geq w} :=
\sum _{j_1 +\cdots+ j_k \geq w} W_{j_1} \wedge \cdots \wedge W_{j_k}\,.
$$

The following lemma is easily checked.
\begin{lemma}
\label{lem_isomPsi}
Let $W$ be a finite dimensional vector space equipped with a filtration \eqref{eq_incrF},
and let $e=(e_1,\ldots, e_n)$ be a basis of $W$ adapted to the filtration. 
Define the linear map 
$$
\Psi : \wedge^k W\to \wedge^k \gr\,( W),
$$
via
$$
\Psi (e_{i_1} \wedge \cdots \wedge  e_{i_k})
= \langle e_{i_1}\rangle \wedge \cdots \wedge  \langle e_{i_k} \rangle, \quad i_1<\ldots <i_n.
$$
Then $\Psi$ is a linear isomorphism that respects the filtration \eqref{eq_filtrationwedgekW}
 of $\wedge^k W$ and the one induced by the gradation of $\wedge^k \gr\, (W)$.
 Moreover, 
 $\gr \, (\Psi)$ provides an isomorphism between 
$$
\gr (\wedge^k W) = \oplus_{w=0}^{k \cdot s} 
W^{ k, \geq w}/  W^{ k, \geq w+1}
$$
and $\wedge^k \gr\, (W)$ respecting their gradations. 
\end{lemma}

Note that the isomorphism $\Psi$ depends on the choice of the basis $e$.

\subsubsection{Euclidean filtered vector spaces}
\label{subsec_Euclgr}
Here, we consider a Euclidean space, that is, a finite dimensional vector space $W$ equipped with a scalar product $(\cdot,\cdot)_W$.

Any subspace $W'\subset W$ is naturally Euclidean, 
its scalar product $(\cdot,\cdot)_{W'}$ being obtained by restricting
$(\cdot,\cdot)_W$.
Moreover, its orthogonal complement $W'^{(\perp, W)}$ in $W$ is naturally identified with the quotient 
$$
W/W' \cong W'^{(\perp, W)}.
$$
This quotient $W/W'$ is also Euclidean if we equip it with the scalar product 
$(\cdot,\cdot)_{W/W'}$ corresponding to
$(\cdot,\cdot)_{W'^{(\perp, W)}}$.

\medskip

Beside being Euclidean, we further assume the vector space $W$ also be filtered with filtration 
\eqref{def strict filtration}.
The resulting graded space $\gr \, W$ inherits a natural scalar product $(\cdot,\cdot)_{\gr W}$ by imposing that the decomposition \eqref{eq_gr_incrF} is orthogonal, 
and that $(\cdot,\cdot)_{\gr W}$ restricted to each $W_{w_j}/W_{w_{j-1}}$
is given by $(\cdot ,\cdot)_{W_{w_j}/W_{w_{j-1}}}$.
Indeed, each $W_{w_j}/W_{w_{j-1}}$ is naturally isomorphic to the orthogonal complement  
$V_j:=
W_{n_{j-1}}^{(\perp, W_{n_j})}$ of $W_{w_{j-1}}$ in $W_{w_{j}}$.
Therefore, by construction, the gradation by quotients on $\gr \, (W)$ is  isomorphic to the following gradation of $W$:
$$
W=V_1\oplus^\perp  V_2 \oplus^\perp \ldots \oplus^\perp V_{s_0}, 
$$
which we will refer to as the \emph{orthogonal gradation} or orthogonal graded decomposition of $W$ (associated with the given scalar product and filtration).

We observe that when considering a basis $e =(e_1,\ldots, e_n)$ of $W$ adapted to its filtration, after a Graham-Schmidt process, we may assume that 
$(e_1,\ldots, e_{n_1})$ is an orthonormal basis of $V_1=W_1$, 
$(e_{n_1+1},\ldots, e_{n_2+n_1})$ is an orthonormal  basis of $V_2$, etc.
In other words, after applying a Graham-Schmidt process, we can always obtain a basis $e$ adapted to the given orthogonal gradation of $W$.

\subsection{Preliminaries on manifolds}
\label{sec_prelM}

\subsubsection{The maps $d$ and $\dalg$ on a general manifold}
\label{subsec_ddalg}
In this section, we recall briefly some well-known facts about the de Rham differential which are valid on any smooth manifold $M$ (without any further assumptions). 
As is customary,  $d$ denotes the exterior differential on the space of smooth forms $\Omega^\bullet(M)$ on $M$.
It is well known that the map $d\colon\Omega^k(M)\to\Omega^{k+1}(M)$ is a smooth differential map that satisfies the Leibniz property and $d^2=0$. 
The associated cochain complex $(\Omega^\bullet(M),d)$ is traditionally referred to as the de Rham complex. 

The explicit formula for $d$ is given for  any  $\omega \in \Omega^k(M)$ and $V_0,\ldots,V_k \in \Gamma(TM)$ by
\begin{align*}
d\omega(V_0,\ldots,V_k)
&=
\sum_{i=0}^k (-1)^i V_i\left(\omega(V_0,\ldots,\hat V_i,\ldots, V_k) \right)
\\&\qquad +
\sum_{0\leq i<j\leq k}
(-1)^{i+j} \omega([V_i,V_j],V_0,\ldots,\hat V_i,\ldots,\hat V_j,\ldots, V_k),
\end{align*}
 the hats denoting omissions.

 In this paper, the algebraic part of $d$ is denoted by  $\dalg$.
 This is the map  given by $\dalg=0$ on $\Omega^0 (M)$, and for $k>0$ and any  $\omega \in \Omega^k(M)$, $V_0,\ldots,V_k\in \Gamma(TM)$, by
\begin{equation}
	\label{eqdef_d0}
	\dalg\omega(V_0,\ldots,V_k)
=
\sum_{0\leq i<j\leq k}
(-1)^{i+j} \omega([V_i,V_j],V_0,\ldots,\hat V_i,\ldots,\hat V_j,\ldots, V_k).
\end{equation}

\begin{lemma}
\label{lem_d0}
	The map $\dalg:\Omega^k (M)\to \Omega^{k+1} (M)$ is 
	smooth and algebraic. It satisfies the Leibniz property and $(\Omega^\bullet(M),\dalg)$ is a complex, i.e. $\dalg^2 =0$.
The action of $\dalg$ over $\Omega^1 (M)$ is given by the explicit formula
$$
\dalg\omega(V_0,V_1) = -\omega ([V_0,V_1])\ ,\  \forall\, \omega\in \Omega^1 (M), \ V_0,V_1\in \Gamma(TM)\,.
$$
\end{lemma}
\begin{proof}[Sketch of the proof]
	Since $\dalg=d-(d-\dalg)$ and for any $\omega \in \Omega^k(M),\,V_0,\ldots,V_k \in \Gamma(TM)$\begin{equation}
	\label{eq_d-d0}
		(d-\dalg)\,\omega(V_0,\ldots,V_k)
=
\sum_{i=0}^k (-1)^i V_i\left(\omega(V_0,\ldots,\hat V_i,\ldots, V_k) \right) \ ,
	\end{equation}
	the Leibniz property of $\dalg$ follows from the fact that both vector fields and $d$ satisfy the Leibniz property.

The formula on 1-forms is easily computed from \eqref{eqdef_d0} by taking $k=1$. 
This formula, together with the Jacobi identity for the commutator bracket, implies $\dalg^2=0$ on $\Omega^1 (M)$. Recursively and by the Leibniz rule, we get that $\dalg^2=0$ on forms of any degree,  since for any $\alpha\in\Omega^k(M)$ and any $\beta\in\Omega^\bullet(M)$
\begin{align*}
    \dalg^2 (\alpha\wedge \beta) =&\dalg\big(\dalg\alpha\wedge\beta+(-1)^k\alpha\wedge\dalg\beta\big) \\=&(\dalg^2 \alpha)\wedge \beta
+(-1)^{k+1} \dalg\alpha\wedge\dalg  \beta+(-1)^k\dalg\alpha\wedge\dalg\beta+\alpha\wedge\dalg^2\beta=0\,.
\end{align*}
\end{proof}

\subsection{Some natural concepts}

In this subsection, we recall some basic concepts in differential geometry.

\subsubsection{Frame and coframe}\label{frame and coframe def}
 Let $E$ be a (real, finite dimensional, and smooth) vector bundle over a (smooth) manifold $M$ with $\dim M=n$.
	A \emph{frame} for $E$ is a smooth section $(S_1, \ldots, S_n)$ of $E\times \ldots \times E=E^n$ such that $(S_1(x),\ldots , S_n(x))$ is a basis of $E_x$ for every $x\in M$.
 A \emph{coframe} for $E$ is frame for the dual vector bundle $E^*$.
 
Throughout this paper, if the bundle $E$ is not specified, we will take $E$ to be the tangent bundle $TM$ over the smooth manifold $M$ and $E^*$ to be the cotangent bundle.
Frames for $TM$ may not exist globally on the whole manifold $M$, but they can always be constructed locally, i.e. over an open neighbourhood of every point in $M$.

\subsubsection{Algebraic maps on bundles}
\label{subsubsec_def_alg}
Let $E$ and $F$ be two smooth vector bundles over a manifold $M$. 
We say that a smooth map $T:\Gamma(E) \to \Gamma(F)$ is linear and algebraic when, for each $x\in M$, there exists a linear map $(T)_{x}: E_x  \to F_x$
satisfying 
$T(\alpha )(x)= (T)_{x}(\alpha (x))$  for any $\alpha\in E$, $x\in M$.
We will call $(T)_{x}$ the pointwise restriction of $T$ above $x\in M$.

\begin{remark}
    \label{rem_eqdef_alg}
 A (smooth) differential operator of order 0 is an algebraic linear map.
\end{remark}

\subsubsection{Metric on bundles}

Let $E$ be a vector bundle over a manifold $M$. 
By definition, a metric $g$ on $E$ is a family of 
 scalar products $g_x :=(\cdot, \cdot)_{E_x}$  on $E_x$ depending smoothly on $x\in M$. The smooth dependence means that 
 $$
 g  \in \Gamma ({\rm Sym}(E \otimes E))\,.
 $$
When $E$ is the tangent bundle $TM$ of a smooth manifold $M$, the metric is said to be Riemannian and
the couple $(M,g)$ is referred to as a Riemannian manifold.

\subsection{Filtered tangent bundles}
\label{sec_filtrationTM}

In this section, we  
consider an $n$-dimensional smooth manifold $M$ whose tangent bundle $TM$ is filtered  by vector subbundles 
\begin{equation}
    \label{eq_filtredTM}
M\times \{0\} = H^0 \subseteq H^1 \subseteq \ldots \subseteq H^s =TM\,.
\end{equation}
Our aim is to extend the concepts of Section \ref{sec_filtrationsVS} to this setting.
After an appropriate choice of indices like in \eqref{def strict filtration}, one can define the smooth vector bundle 
$$
\gr( TM) =  \oplus_{i=1}^{s} H^{i}/H^{i-1}=\oplus_{j=1}^{s_0}H^{w_j}/H^{w_{j-1}}\ , 
$$
with $n_j := \dim H^{w_j}/H^{w_{j-1}}>0, \ j=1,\ldots, s_0$, and $dim(M)=n=n_1+n_2+\cdots+n_{s_0}$.
\subsubsection{Adapted frames}
\label{subsec_TMfiltered}

The notion of adapted bases in the setting of filtered vector spaces naturally leads to the notion of adapted frames and coframes.

\begin{definition}
A frame $\bX=(X_1,\ldots, X_n)$ for $TM$
is said to be \emph{adapted to the filtration} \eqref{eq_filtredTM} when, for every $x\in M$, the basis $(X_1(x),\ldots,X_n(x))$ is adapted to the filtration  $T_x M =H^{s_0} _x\supseteq \ldots \supseteq H^{w_1}_x \supseteq H^{w_0}_x=\{0\}$.
	\end{definition}
	
If $\bX=(X_1,\ldots, X_n)$ is a frame for $TM$ adapted to the filtration \eqref{eq_filtredTM},
then the associated graded basis on $\gr\,(T_x M)$
$$
\langle \bX\rangle_x = (\langle X_1\rangle_x,\ldots, \langle X_n \rangle_x)\,, \ x\in M\,,
$$
 yields a frame $\langle \bX\rangle$  for $\gr (TM)$  adapted to the gradation in the following sense.
\begin{definition}
	A frame $\bY = (Y_1,\ldots , Y_n)$ for $\gr( TM)$ is \emph{adapted to the gradation}, or \emph{graded}, when
	 $Y_1,\ldots, Y_{n_1}$ is a frame for $H^{w_1}/H^{w_0}$, 
	 $Y_{n_1+1},\ldots, Y_{n_2+n_1}$ is a frame for $H^{w_2}/H^{w_1}$, and so on.
\end{definition}

By duality, we obtain the following decreasing filtration of the cotangent bundle $T^*M$ by vector bundles
\begin{equation}
	\label{eq_filtrationTM*}
T^*M = (H^{w_0})^\perp \supseteq (H^{w_1})^\perp \supseteq \ldots
\supseteq (H^{w_{s_0}})^\perp = M\times \{0\}\,,
\end{equation}
by considering the annihilators $(H^{w_j})^\perp$ of $H^{w_j}$ for $j=1,\ldots,s_0$, as follows.

\begin{definition}
A coframe $\Theta$ for $TM$
 is \emph{adapted to the filtration} when, for every $x\in M$, 
 the basis $\Theta(x):=(\theta^1(x), \ldots, \theta^n(x))$ is adapted to the filtration of $T_x^* M = (H^{w_0}_x)^\perp \supseteq (H^{w_1}_x)^\perp \supseteq \ldots
\supseteq (H^{w_{s_0}}_x)^\perp = \{0\}$.
\end{definition}
This definition is equivalent to saying that, at each point $x\in M$, the final $n_{s_0}$ covectors annihilate $H^{w_{s_0-1}}_x$. Furthermore, these final $n_{s_0}$ covectors together with the preceding $n_{s_0-1}$ covectors annihilate $H^{s_0-2}_x$, and so on.
The associated frame
 $\langle \Theta\rangle=(\langle \theta^1\rangle, \ldots,\langle \theta^n\rangle)$ is a frame for  $\gr(T^* M)$ that is adapted to the gradation in the following sense.
 \begin{definition}
	A frame $\Xi = (\xi^1,\ldots , \xi^n)$ for $\gr( T^*M)$  is  said to be \emph{adapted to the gradation} (or  simply \emph{graded}) when 
	$\xi^{n},\ldots, \xi^{n-n_{s_0} +1}$ is a frame for $(H^{w_{s_0-1}})^\perp /  (H^{w_{s_0}})^\perp$,  
	 $\xi^{n - n_{s_0}},\ldots,$ $ \xi^{n-n_{s_0-1}+1}$ is a frame for $(H^{w_{s_0-2}})^\perp/(H^{w_{s_0-1}})^\perp$, and so on.
\end{definition}

One can readily check that 
if $\bX$ is a frame adapted to a gradation for $TM$, then its dual $\Theta$  is an adapted  coframe for $TM$.
Conversely, if $\Theta$ is a coframe for $T^\ast M$ adapted to a gradation of $TM$, then its dual $\Theta^\ast=\bX$ is a graded frame for $TM$. 

As mentioned in Subsection \ref{frame and coframe def}, frames for $TM$ can always be constructed locally. Furthermore, through a pivoting process, one may easily construct a local frame for $TM$ adapted to a given gradation.
The inverse function theorem implies that, given a local frame $\bY$ of $\gr(TM)$ adapted to the gradation, we can then construct a local frame $\bX$ of $TM$ adapted to the filtration such that $\langle \bX \rangle =\bY$.

\subsubsection{Weights of forms}

Building on top of Section \ref{subsec_TMfiltered},
we define below the natural notion of forms of weights at least $w$ in this context.

Given an increasing filtration \eqref{eq_filtredTM} of $TM$, we denote by $H^j$ the set of (based) vectors of weight at most $j$ (shortened with $\leq j$).
By duality, we say that the (based) covectors in $(H^{j-1})^\perp $ have weight at least $j$ (shortened as $\geq j$).
We extend this vocabulary to the space of smooth sections $\Gamma(TM)$ and $\Gamma (T^*M) =\Omega^1 (M)$, that is vector fields and 1-forms respectively.

With our definition, a 1-form $\alpha\in \Omega^1 (M)$ is of weight $\geq w$ when $\alpha (V) = 0$ for any $V\in \Gamma(H^{w-1})$.
We denote the space of 1-forms of weight at least $j$ by
$$
\Omega^{1,\geq j}  (M):= \Gamma((H^{j-1})^\perp), \qquad j=1,2,\ldots, s_0+1\,.
$$
The filtration in \eqref{eq_filtrationTM*} of the cotangent bundle then determines the following strict filtration  
$$
\Omega^1 (M)=\Omega^{1,\geq w_1}( M)
\supsetneq
\Omega^{1,\geq w_2}(M)
\supsetneq \ldots \supsetneq
\Omega^{1,\geq w_{s_0+1}}(M) =\{0\}.
$$
\begin{ex}\label{first example}
	If $\Theta:=(\theta^1, \ldots , \theta^n)$ is a frame for $T^\ast M$ adapted to the filtration \eqref{eq_filtrationTM*}, then all the 1-forms $\theta^1,\ldots,\theta^{n}$ are of weight $\geq w_1$. The 1-forms $\theta^{n_1+1},\ldots,\theta^{n}$ are of weight $\ge w_2$, and more generally
$\theta^{n_1+\cdots+n_{i-1} +1},\ldots,  \theta^{n}$ are of weight $\geq w_i$. By taking $n_0=\dim H^{w_0}=0$, we can state this as
$$
\theta^j \in\Omega^{1,\geq w_i} (M), \quad j=n_0+\cdots +n_{i-1}+1,\ldots,n\ , \quad 
i=1,\ldots, s_0+1.
$$
\end{ex}

In general, we will not have a global coframe as in Example \ref{first example}, but just local ones.

For $k=0$, we set 
$$
C^\infty (M) = \Omega^0 (M) := \Omega^{0,\geq 0}(M) ,
\qquad\mbox{and}\qquad
\Omega^{0,\geq 1} (M) :=\{0\}.
$$

For $k=1,2,\ldots$, we  define the space of $k$-forms of weight at least $w$ (shortened as $\geq w$) by
$$
\Omega^{k,\geq w}(M)
:=
\oplus _{i_1+\ldots+i_k \geq w} \Omega^{1,\geq i_1}(M)\wedge \ldots \wedge 
 \Omega^{1,\geq i_k}(M).
 $$
 In other words, a $k$-form is of weight at least $w$ when it can be written locally as a linear combination of $k$-wedges of 1-form of weights at least $i_1,\ldots,i_k$ with $w\leq i_1+\ldots+i_k$. 
 From the definition, it follows that a $k$-form $\alpha\in \Omega^k(M)$ is of weight $\geq w$ when 
 $$
 \forall\ V_1\in \Gamma(H^{i_1}), \ldots, V_k\in \Gamma(H^{i_k})\qquad
 i_1+\ldots+i_k <w \Longrightarrow
 \alpha (V_1,\ldots ,V_k) = 0.
 $$

\begin{ex}
\label{ex_Thetageqweight}
Consider an adapted coframe
$\Theta$ for $TM$.
We will often use the shorthand 
$$
\Theta^{\wedge I} = \theta^{i_1}\wedge \ldots \wedge \theta^{i_k}
$$
for the multi-index $I=(i_1,\ldots,i_k)$, always assuming $i_1<\ldots <i_k$. 
Then $\Theta^{\wedge I} $  is of weight $\geq \upsilon_{1}+\ldots +\upsilon_{k}$, where each $\theta^{i_j}$ is of weight $\ge \upsilon_j$.
\end{ex}

\begin{definition}\label{def free basis k forms}
The family $(\Theta^I)_I$ where $I$ runs above all multi-indices of length $k$ (taken with strictly increasing indices as in example \ref{ex_Thetageqweight} above) is a free basis of the $C^\infty(M)$-module $\Omega^\bullet (M)$   called the basis associated with the coframe $\Theta$.
\end{definition}

By construction, we have the inclusion 
$\Omega^{k,\geq w} (M) \wedge  \Omega^{j,\geq w'}(M)
\subseteq \Omega^{k+j,\geq w+ w'}(M)$ for any $k,j\ge 0$. Moreover, given the increasing filtration \eqref{def strict filtration} of $TM$, for any degree $k\ge 0$
$$
\Omega^{k,\geq k \cdot w_{s_0}+1}(M) =\{0\}\ \text{ and }\ w\leq w' \ \Longrightarrow \
\Omega^{k,\geq w} (M) \supseteq \Omega^{k,\geq w'} (M)\,.
$$
In particular, the module $\Omega^k(M)$ over the ring $C^\infty (M)$ admits the filtration 
\begin{align}\label{filtration on k forms by weight}
    \Omega^k(M) = \Omega^{k,\geq 0}  (M)
\supset \Omega^{k,\geq 1} (M)\supset \ldots \supset
\Omega^{k,\geq w} (M)\supset \ldots \supset \Omega^{k, \geq k\cdot w_{s_0}+1} (M)=\{0\}.
\end{align}

The basis $(\Theta^I)_I$ introduced in Definition \ref{def free basis k forms} is adapted to this filtration. 

\subsubsection{The bundle $\wedge^k \gr( T^*M)$ and its weights}
\label{subsubsec_wedgekgrTM*}

Given the filtration \eqref{eq_filtredTM}, by duality the smooth vector bundle 
$$ 
\gr (T^*M) 
= \oplus_{i=1}^s  ((H^{i-1})^\perp / (H^{i})^\perp )=
\oplus_{j=1}^{s_0}((H^{w_j-1})^\perp / (H^{w_j})^\perp )$$
is graded.
We say that an element of $\gr (T^*M)$  is of weight $j$ when it is in $(H^{j-1})^\perp / (H^{j})^\perp$.
We extend the same notion of weight to smooth sections of $\gr(T^*M)$. 

\begin{ex}
\label{ex_Theta_1formweight}
	Let us consider the same adapted coframe
$\Theta$ for $TM$ as in Example \ref{first example}.
Then $\langle \theta^{1}\rangle,\ldots,  \langle \theta^{n_1}\rangle$ have weight $w_1$, and for $i=2,\ldots,s_0$, 
the sections
 $\langle \theta^{n_1+\cdots+n_{i-1} +1}\rangle,\ldots, $ $ \langle \theta^{n_1+\cdots+n_{i}}\rangle$ of $\gr (T^*M)$ have weight $w_i$.
 If we denote by $\upsilon_j$ the  weight of $\langle \theta^j\rangle$, we have: 
\begin{equation}
	\label{eq_defupsilon_j}
	\upsilon_{n_1+\cdots+n_{i-1}+1} = \ldots = \upsilon_{n_1+\cdots+n_i} = w_i, 
	\qquad i=1,\ldots,s_0.
\end{equation}
 \end{ex}

There is a natural identification between the following smooth vector bundles
$$
\wedge^k \gr(T^*M) := \cup_{x\in M} \wedge^k \gr(T_x^*M) 
\cong \cup_{x\in M} (\wedge^k \gr(T_xM) )^* = :\wedge^k \gr(TM)^*,
$$
allowing us to consider directly $\wedge^k \gr (T^*M) $ in our discussion. 

It is straightforward to check that $\wedge^k \gr (T^*M) $ is  a module over $C^\infty(M)$ which, by Section \ref{subsubsec_gradation},
is naturally equipped with a gradation and a notion of weight. In other words, 
$$
\wedge^k \gr (T^*M)
= \oplus_{w\in \bN} \wedge^{k,w} \gr (T^*M),
$$
where the elements of weight $w$ are in the linear subbundle
$$
\wedge^{k,w} \gr (T^*M) := \sum_{i_1+\ldots + i_k=w} 
((H^{i_1-1})^\perp / (H^{i_1})^\perp )
\wedge \ldots \wedge
((H^{i_k-1})^\perp / (H^{i_k})^\perp )
.$$
Once we extend this construction to the space of smooth sections of $\wedge^\bullet\gr\,(T^\ast M)$, we get the gradation
\begin{align}\label{gradation on gr k forms by weights}
    \Gamma(\wedge^\bullet\gr\,(T^\ast M))=\oplus_{k,w\in\mathbb N_0}\ 
 \Gamma(\wedge^{k,w}\gr\,(T^\ast M))\,,
\end{align}
where the $k$-forms of weight $w$ are given by 
\begin{align*}
    \Gamma(\wedge^{k,w}\gr\,(T^\ast M))=&\sum_{i_1+\cdots+i_k=w}\Gamma((H^{i_1-1})^\perp)/\Gamma((H^{i_1})^\perp)\wedge\cdots\wedge\Gamma((H^{i_k-1})^\perp)/\Gamma((H^{i_k})^\perp)\,.
\end{align*}

\begin{ex}
\label{ex_Theta_kformweight}
We continue with the setting of Example \ref{ex_Theta_1formweight} and examine the case of $k$-forms, $k>1$. 
The sections
 $\langle \theta^{i_1}\rangle \wedge \ldots \wedge \langle \theta^{i_k}\rangle$ of $\wedge^k \gr (T^*M)$ are of weight 
 $$
 |(\upsilon_1,\ldots,\upsilon_k)|:=\upsilon_{i_1}+\ldots +\upsilon_{i_k}.
 $$
\end{ex}
We may use the shorthand notation $\langle\Theta\rangle^{\wedge I} = \langle\theta^{i_1}\rangle\wedge \ldots \wedge \langle\theta^{i_k}\rangle$ where $I$ is the multi-index $I=(i_1,\ldots,i_k)$ with $i_1<\ldots<i_k$. 
 
We readily check the following properties:
\begin{lemma}
\label{lem_Theta_corr}
The family $(\langle \Theta\rangle^I)_I$, where $I$ runs above all multi-indices of length $k$ and strictly increasing indices, is a basis of $\Gamma(\wedge^k \gr(T^*M))$ adapted to the gradation.
Moreover, it is the basis associated to the  basis $(\Theta^I)_I$ adapted to the filtration \eqref{filtration on k forms by weight} of $\Omega^\bullet (M)$
in the sense described in Section \ref {subsubsec_anglee}.
\end{lemma}

\subsubsection{The $\gr$ operation}

\begin{definition}
A morphism 
 $D:\Omega ^\bullet (M)\to \Omega^\bullet (M)$  of $C^\infty(M)$-modules  respects the filtration \eqref{filtration on k forms by weight}, 
	when
$$
\forall  w\in \bN_0
\qquad D\, \Omega^{\bullet,\geq w} (M) \subseteq \Omega^{\bullet,\geq w} (M).
$$ 
\end{definition}

The following property follows readily 
from Section \ref{sec_filtrationsVS}:
\begin{lemma}
\label{lem_gr_modulemorphism}
	Let $D:\Omega ^\bullet (M)\to \Omega^\bullet (M)$ be a morphism of $C^\infty(M)$-modules that respects the filtration \eqref{filtration on k forms by weight}. Then the map 
$$
\gr (D) :\Gamma ( \wedge^\bullet  \gr(T^*M))
\to \Gamma ( \wedge^\bullet  \gr(T^*M)) ,
$$
is a morphism of $C^\infty(M)$-modules that respects the gradation \eqref{gradation on gr k forms by weights}.
 If in addition, $D$ is a differential operator, then $\gr(D)$ is also a differential operator of the same order or lower.

As an operation, $\gr$ is a module morphism from  $\{D:\Omega ^\bullet (M) \to \Omega^\bullet (M)  :  $ module morphism$\}$
to $\{T:\Gamma ( \wedge^\bullet  \gr(T^*M)) \to \Gamma ( \wedge^\bullet  \gr(T^*M))  :  $ module morphism$\}$.

\end{lemma}

By construction,  $\gr(D)=0$ if and only if $D$ increases the weight in the following sense.
\begin{definition}
	We say that a  linear map $D:\Omega ^\bullet (M) \to \Omega^\bullet (M)$ \emph{increases the weights} of  the filtration when 
$$
\forall  w\in \bN_0
\qquad D \,\Omega^{\bullet,\geq w} (M) \subseteq \Omega^{\bullet,\geq w+1} (M).
$$ 
\end{definition}
As there are only a finite number of weights, a linear map that increases the weights is necessarily nilpotent, that is
\begin{equation}
	\label{eq_N_0}
	 \gr \, D =0 \Longrightarrow \quad \exists\, N_0\in \bN
 \ \text{ such that } D^{N_0} =0.
\end{equation}
This $N_0$ depends only on $M$ and its filtration \eqref{eq_filtredTM} (via the filtrations \eqref{filtration on k forms by weight}), not on $D$. Throughout this paper, we will assume it to be a fixed integer.

\subsection{Tangent bundles equipped with a filtration and a metric}
\label{subsec_filtration+metric}
Let $M$ be a manifold whose tangent bundle admits a filtration \eqref{eq_filtredTM} by subbundles and is equipped with a metric $g_{TM}$.

\subsubsection{Identification of $\Omega^\bullet (M)$ and $\Gamma(\wedge^\bullet \gr(T^*M))$}
\label{subsubsec_natorthgrad}
\label{subsubsec_PhiM}

The considerations in Section \ref{subsec_Euclgr} show that $\gr (TM)$ inherits a metric
$g_{\gr(TM)}$.
The same section also implies
that, at each point $x\in M$, there exists an open neighbourhood $U$ of $x$ and a frame $\bX$ of $TM$  on  $U$ that is $g_{TM}$-orthonormal  and adapted to  the natural orthogonal gradation 
$$
 TM = \oplus_j^\perp G^j,
$$
where
$G^j$ is the orthogonal complement of $H^{w_{j-1}}$ in  $H^{w_j}$ for each $j=1,\ldots,s_0$.

Let $x\in M$. 
For any $j=1,\ldots,s_0$, if $\xi$ is a covector in $G^j_x$ (i.e. a linear form on $T_x M$ vanishing on every $G^k_x$, $k\neq j$), we denote by $\Phi_x(\xi)$ the corresponding covector in $H^{w_j}_x/H^{w_{j-1}}_x$ (i.e. the corresponding linear form on $\gr(T^*_x M)$  vanishing on $H^{w_k}_x/H^{w_{k-1}}_x$, $k\neq j$).
 This extends uniquely to a linear bijective map $\Phi_x:T_x^*M\to \gr(T_x^*M) $, 
 and then to a $C^\infty(M)$-module isomorphism 
 $$
 \Phi : \Omega^\bullet (M)\to \Gamma(\wedge^\bullet \gr(T^*M))
 $$
 respecting the Leibniz property.

\begin{lemma}
\label{lem_PhiM}
Let $M$ be a manifold whose tangent bundle admits a filtration by subbundles \eqref{eq_filtredTM} and is equipped with a metric $g_{TM}$.
The map $\Phi$ constructed above
is an isomorphism of $C^\infty(M)$-modules
that respects the filtrations, even when restricted 
to $\Omega^k(M) \to \Gamma(\wedge^k \gr(T^*M))$ for each $k=0,1,\ldots,n$.
Moreover, $\gr ( \Phi)$
is an isomorphism between the $C^\infty(M)$-modules
$\gr(\Omega^k(M) )$ and $\Gamma(\wedge^k \gr(T^*M))
$ that respects the gradation. 

If $\bX$  is a $g_{TM}$-orthonormal local frame of $TM$ on an open subset $U\subset M$ that is adapted to the natural orthogonal gradation $ TM = \oplus^\perp G^j$, 
then we have
$$
\Phi(\Theta^I)=\langle\Theta\rangle^I
$$
for any increasing  multi-index $I$ of length $k$, where $\Theta$ is the coframe associated to $\bX$ over $U$.
\end{lemma}

Note that the map $\Phi$ is in fact algebraic, with 
$$
\Phi_x :\Gamma(\wedge^\bullet T_x^*M) \to \Gamma(\wedge^\bullet \gr(T^*M))
\quad\mbox{defined  via} \quad
\Phi_x (\Theta^I(x)) = \langle \Theta\rangle^I(x),
$$
using the notation of Lemma \ref{lem_PhiM}.
When restricted to $\Omega^1(M)$, the matrix representing the map $\Phi$  is just the identity matrix of $\bR^n$:
$$
	\Mat_{\Theta}^{\langle \Theta\rangle} (\Phi:\Omega^1(M) \to \Gamma(\gr(T^\ast M))) = \id.
$$
Moreover, for each $k=2,\ldots,n$, using the notation $\Theta^k = (\Theta^I)_I$ and $\langle \Theta\rangle^k = (\langle \Theta\rangle^I)_I$ for the corresponding  bases of $\Omega^k (M)$ and of $\Gamma(\wedge^k \gr(T^*M))$, 
the matrix representing $\Phi\colon\Omega^k(M)\to\Gamma(\wedge^k\gr(T^\ast M))$ is the identity matrix of $\bR^{\dim \Omega^k( M)}$:
$$
	\Mat_{\Theta^k}^{\langle \Theta\rangle^k} (\Phi) = \id\,.
$$

Intertwining with $\Phi$ provides forms with a reverse operation to the operation $\gr$ described  in Lemma \ref{lem_gr_modulemorphism}:
\begin{lemma}
\label{lem_Phi}
We continue with the setting of Lemma \ref{lem_PhiM}.
	\begin{enumerate}
		\item If  $T: \Gamma(\wedge^\bullet \gr(T^*M)) \to \Gamma(\wedge^\bullet \gr(T^*M)) $ is a morphism of $C^\infty(M)$-modules, 
then 
$$
T^{\Phi}:=\Phi^{-1}\circ T\circ \Phi :\Omega^\bullet (M) \to \Omega^\bullet (M)
$$
	is a morphism of $C^\infty(M)$-modules.
 If in addition $T$ is a differential operator, then $T^\Phi$ is also a differential operator of the same order. \label{part 1, lemma 4.16}
	\item If  $T:\Gamma(\wedge^\bullet \gr(T^*M)) \to \Gamma(\wedge^\bullet \gr(T^*M))$ is a morphism of $C^\infty(M)$-modules that  respects the gradation, 
then 
$T^{\Phi} :\Omega^\bullet (M) \to \Omega^\bullet (M)$ respects the filtration and we have
$$
\gr (T^{\Phi}) = T.
$$\label{part 2, lemma 4.16}
\item \label{part 3, lemma 4.16}
If $D:\Omega^\bullet (M) \to \Omega^\bullet (M)$ respects the filtration, then 
$D- ( \gr( D) )^\Phi $ strictly increases weights, and so 
$$
\gr( D)= \gr\left(  ( \gr( D) )^\Phi\right).
$$
	\end{enumerate}
\end{lemma}
\begin{proof}
	These properties are easily checked on $\Theta^I$ and $\langle \Theta\rangle^I$ having fixed a local orthonormal adapted coframe $\Theta$.
\end{proof}

\subsubsection{Scalar products on $\wedge^\bullet T_x^*M$ and of $\wedge^\bullet \gr(T_x^*M)$}
\label{subsubsec_scpdt}
For each $x\in M$, the scalar products $g_{T_x M}$ and $g_{\gr(T_x M)}$  on $T_x M$ and $\gr (T_x M)$ respectively induce  scalar products on their  duals $T_x^* M$ and $\gr(T_x M)^*\cong\gr(T_x^\ast M)$, and then scalar products $g_{\wedge^\bullet T_x ^*M}$ on $\wedge^\bullet T_x^* M$ and 
$g_{\wedge^\bullet \gr(T_x^* M)}$ on $\wedge^\bullet \gr(T_x^* M)$
via the formula in \eqref{eq_scpdt_det} below and 
the orthogonal decompositions:
$$
\wedge^\bullet T_x^* M = \oplus^\perp \wedge^k T_x^* M
\quad\mbox{and}\quad
\wedge^\bullet \gr(T_x^* M) = \oplus^\perp \wedge^k \gr(T_x^* M)\,.
$$
We will denote all these scalar product as $\langle\cdot,\cdot\rangle$, their meaning being clear from the context or from an indexation.
We have 
\begin{equation}
	\label{eq_scpdt_det}
	\langle \alpha_1\wedge \ldots\wedge \alpha_k,  \beta_1\wedge\ldots\wedge \beta_k\rangle_{\wedge^k T_x^*M} := \det\left( \langle \alpha_i,\beta_j\rangle_{T_x^*M}\right)_{1\leq i,j\leq k},
\end{equation}
where 
$\alpha_1,\ldots,\alpha_k,\beta_1\ldots,\beta_k\in T_x ^*M$, 
and similarly for 
$ \gr(\wedge ^kT_x^*M)\cong \wedge^k\gr(T_x^\ast M)$.

We check readily that 
the map $\Phi_x$ defined above sends the scalar product of  $\wedge^\bullet T_x^* M$ onto the one of $\wedge^\bullet \gr(T_x^* M)$:
\begin{equation}
\label{eq_scalarpdtPhi}
\Phi_x\colon\wedge^\bullet T_x^\ast M\to \wedge^\bullet\gr(T_x^\ast M)\ ,\quad 
    \langle \Phi_x (\alpha),\Phi_x(\beta)\rangle = \langle \alpha,\beta\rangle\ \forall\, \alpha,\beta\in \wedge^\bullet T_x^* M\,.
\end{equation}

Given an orthonormal local frame $\bX$ or equivalently an orthonormal local coframe $\Theta$, we check readily that  
the bases $\Theta^k(x)=(\Theta^I)_I(x)$
and $\langle \Theta^k\rangle(x)=(\langle\Theta^I\rangle)_I(x)$ of 
$\wedge^k T_x^*M$ and of $\wedge^k \gr(T_x^*M)$ respectively are orthonormal.
We will use the following notation to denote the elements of the bases in degree 0 and $n$:
$$
1=\Theta^\emptyset \text{ for }\wedge^0T^\ast_xM\ \text{ and }\vol_{\Theta}(x):=\theta^1(x)\wedge \ldots \wedge \theta^n(x) \text{ for }\wedge^nT_x^\ast M
$$
and 
$$
1=\langle \Theta\rangle^\emptyset\text{ for }\wedge^0\gr(T_x^\ast M)\text{ and } 
\langle \vol_{\Theta}\rangle(x) :=
\langle \theta^1\rangle(x)\wedge \ldots \wedge \langle\theta^n\rangle(x)\text{ for }\wedge^n\gr(T_x^\ast M)\,.
$$

With these scalar products, we can now define transposes and obtain some of their properties:
 if $T$ and $D$ are  algebraic maps acting linearly on 
$\wedge^\bullet \gr(T^* M)$ and $\Omega^\bullet (M)$ respectively, then 
we define 
their transpose $T^t$ and $D^t$
as the algebraic maps acting linearly on $\wedge^\bullet \gr(T^* M)$ and $\Omega^\bullet (M)$ by their pointwise transpose, that is,  $(T^t)_x$  and $(D^t)_x$ are the transpose of $T_x$ and $D_x$ for the scalar products $g_{\wedge^\bullet \gr(T_x^* M)}$ and $g_{\wedge^\bullet T_x^* M}$ at each point $x\in M$.
We check readily that 
if $T$ is an algebraic map acting linearly on 
$\wedge^\bullet \gr(T ^*M)$ respecting the gradation,
then so is $T^t$. Moreover,
the map
$D:= \Phi^{-1} \circ T \circ \Phi$
and its transpose $D^t$
are both algebraic map acting linearly on $\Omega^\bullet (M)$  respecting the gradation, and satisfy:
$$
D^t = \Phi^{-1}\circ T^t \circ \Phi 
\qquad\mbox{and}\qquad 
\gr (D^t) = (\gr(D))^t.
$$

\subsubsection{Hodge-star, scalar product on $\Omega_c^\bullet(M)$ and $\Gamma_c(\wedge^\bullet \gr(T^*M))$, and transpose}
\label{subsubsec_star}
The Hodge star operator on $\wedge^\bullet T_x^* M$ or $\wedge^\bullet \gr(T_x^* M)$
is the map defined via 
$$
\star : \wedge^k T_x^* M
 \overset {\cong} \longrightarrow 
 \wedge^{n-k} T_x^* M\,, 
\ 
\alpha\wedge \star\beta := \langle\alpha,\beta\rangle 
\vol_{\Theta}(x)
\ \forall\,\alpha,\beta \in \wedge^k T_x ^*M, \ k=1,\ldots,n\,,
$$
and similarly on $\wedge^\bullet \gr(T_x^* M)$.
As $\Phi_x$ maps the scalar product of $\wedge^\bullet T_x^* M $ to the one of $\wedge^\bullet \gr(T_x^* M)$, (see \eqref{eq_scalarpdtPhi}),  it also maps the Hodge star operator of $\wedge^\bullet T_x^* M $ to the one of $\wedge^\bullet \gr(T_x ^*M)$:
\begin{equation}
\label{eq_*Phi}
\forall \alpha,\beta\in \wedge^\bullet T_x^* M  \qquad
\star \Phi_x (\alpha) = \Phi_x(\star \alpha).
\end{equation}

The $\star$ operator naturally extends to an algebraic operator 
on $\Omega^\bullet (M)$ and on $\Gamma(\wedge^\bullet \gr(T^* M))$
respectively. Moreover, from the definition of $\star$, we have that
$$
\alpha \wedge \star \beta = \beta \wedge \star \alpha \quad \mbox{on} \ 
\Omega^\bullet (M) \ \mbox{and}\ \Gamma(\wedge^\bullet \gr(T^* M)),
$$
and 
$$
\star 1=\vol_\Theta, \ 
\star \vol_\Theta =1, 
\quad\mbox{and}\quad
\star 1=\langle\vol_\Theta\rangle, \ 
\star \langle\vol_\Theta\rangle =1.
$$
Moreover, if a  local orthonormal adapted coframe $\Theta$ is fixed, the Hodge star operators on $\Theta^I$ and $\langle \Theta\rangle^I$  with multi-index $I=(i_1,\ldots,i_k)$
are given by
$$
\star (\Theta^I)
=
(-1)^{\sigma(I)}
\Theta^{\bar I}, 
\qquad 
\star (\langle\Theta\rangle^I)
=
(-1)^{\sigma(I)}
\langle\Theta\rangle^{\bar I}
$$
where $\bar I = (\bar i_1,\ldots,\bar i_{n-k})$ with $\bar i_1<\bar i_2<\ldots <\bar i_{n-k}$ is the multi-index complementing $I$, and $(-1)^{\sigma(I)}$ is the sign of the permutation $\sigma(I)=i_1\cdots i_k\,\bar i_1\cdots\bar i_{n-k}$.
This completely characterises the Hodge star operators on both  $\Omega^\bullet (M)$ and on $\Gamma(\wedge^\bullet \gr(T^* M))$.

The previous properties imply readily that  for any $\alpha,\beta$ in either $\Omega^k (M)$ or $\Gamma(\wedge^k \gr(T^* M))$, we have
$$
\star\star \beta = (-1)^{k(n-k)}\beta
\qquad\mbox{and}\qquad\langle\star \alpha,\star \beta\rangle= \langle\alpha,\beta\rangle\, .
$$

We can now define the so-called $L^2$-inner products on 
the subspaces $\Omega^\bullet_c(M)$ and  $\Gamma_c (\wedge^\bullet \gr(T^*M))$ of forms with compact support in 
 $\Omega^\bullet(M)$ and $\Gamma (\wedge^\bullet \gr(T^*M))$ respectively  via
$$
\langle\alpha,\beta\rangle_{L^2}
:= \int_M \alpha \wedge \star \beta
= \int_M \langle \alpha , \beta\rangle\, \vol
\ ,\ \forall\, 
\alpha, \beta\in \Omega^\bullet_c(M)  \text{ or }
 \Gamma_c (\wedge^\bullet \gr(T^*M))\, .
$$
We have $\Phi (\Omega^\bullet_c(M)) = \Gamma_c (\wedge^\bullet \gr(T^*M))$ and
\begin{equation}
\label{eq_L2scalarproductPhi}
\forall \alpha,\beta\in \Omega^\bullet_c(M)  \qquad
\langle\alpha,\beta\rangle_{L^2}
=
\langle \Phi(\alpha),\Phi(\beta)\rangle_{L^2}.
\end{equation}

\subsubsection{Properties of the transpose for the $L^2$-inner product}
The above definitions of the Hodge operator and of the scalar product show that if $D$ is a differential operator on either $\Omega^\bullet (M)$ or $\Gamma(\wedge^\bullet \gr(T^* M))$, then so are $\star D$ and $D\star$, as well as the formal transpose $D^t$ defined as
$\langle D\alpha,\beta\rangle_{L^2}=\langle\alpha,D^t\beta\rangle_{L^2}$ for any $\alpha\in \Omega_c^k(M)$ or $\Gamma_c(\wedge^k\gr(T^\ast M))$ and any $\beta\in\Omega^{k+1}_c(M)$ or $\Gamma_c(\wedge^{k+1}\gr(T^\ast M))$.

In the particular case of the de Rham differential $d$ on forms, we observe that the transpose of $d^{(k)} = d :\Omega^k(M) \to \Omega^{k+1}(M)$ is a differential operator that satisfies
$$
d^{(k,t)}  = (-1)^{kn +1} \star  d^{(n-k-1)} \star\colon\Omega^{k+1}(M)\to\Omega^k(M)\,.
$$ 
We can also consider the transpose of the algebraic linear map $\tilde d^{(k)} = \tilde d : \Omega^k (M)\to \Omega^{k+1}(M)$ of Subsection \ref{subsec_ddalg}. 
The linear map $\tilde d^t$ is then  algebraic and satisfies
$$
\tilde d^{(k,t)}  = (-1)^{kn +1} \star \tilde d^{(n-k-1)} \star\colon\Omega^{k+1}(M)\to\Omega^{k}(M)\,.
$$ 

\smallskip

It is important to realise that the transpose may not respect the filtration in the following sense:
for a general differential operator $D\colon\Omega^\bullet(M)\to\Omega^\bullet(M)$,
it is true that the transpose  $D^t$ is a differential operator, but   $D$  respecting the filtration \eqref{filtration on k forms by weight} will not imply always that  $D^t$ does so.
This may not be true 
even for $d$ and $\tilde d$,  and two explicit `counter-examples' of this situation are given in Appendix \ref{sec_app_ex}. However, the following lemma  guarantees that $D^t$ will also respect the filtration when $D$ is of the form $D=T^\Phi  = \Phi^{-1} \circ T \circ \Phi$ with $T$  a homomorphism  of the $C^\infty(M)$-module $\Gamma(\wedge^\bullet \gr(T^*M))$:

\begin{lemma}
\label{lem_grDt} We continue with the setting of Lemma \ref{lem_Phi}, where $T:\Gamma(\wedge^\bullet \gr(T^*M)) \to \Gamma(\wedge^\bullet \gr(T^*M))$ is a morphism  of $C^\infty(M)$-modules.
We consider the map 
$D:=T^\Phi  = \Phi^{-1} \circ T \circ \Phi$
 acting on $\Omega^\bullet (M)$.
Let us further assume that $T$ respects the gradation.
\begin{enumerate}
    \item The $C^\infty(M)$-module morphism $\star T\star\colon\Gamma(\wedge^\bullet \gr(T^*M)) \to \Gamma(\wedge^\bullet \gr(T^*M))$ also respects the gradation,
and $ (\star T \star)^\Phi = \star T^\Phi \star.$
Consequently, the $C^\infty(M)$-module morphism $D:=T^\Phi:\Omega^\bullet (M) \to \Omega^\bullet (M) $ respects the filtration and we have $\gr (\star D \star) =\star \gr(D) \star  $.
\label{part 1, Lemma 4.17}
\item
If in addition, $T$ is a differential operator, then $T^t$ is also a differential operator of the same order as $T$, it respects the gradation, and satisfies
$(T^t)^\Phi = (T^\Phi)^t$.
Consequently, the differential operators on smooth forms $D:=T^\Phi$ and its transpose $D^t$ respect the filtration, and $\gr (D^t) = (\gr(D))^t.$\label{part 2, Lemma 4.17}
\end{enumerate}
\end{lemma}

\begin{proof}
The equality $ (\star T \star)^\Phi = \star T^\Phi \star$ follows from \eqref{eq_*Phi}.

Let $\Theta$ be a local adapted frame. 
We observe that $\langle \Theta \rangle^I$ is of weight $|I|$
while $\star \langle \Theta \rangle^I = (-1)^{\sigma(I)} \langle \Theta \rangle^{\bar I}$ is of weight $|\bar I| = Q-|I|$ where  $Q:= \upsilon_1+\ldots + \upsilon_n$ is the weight of the volume form.

Since $T$ respects the gradation and the weight of $\langle \Theta\rangle ^{\bar I}$ is $|\bar I|$, $T (\langle \Theta\rangle ^{\bar I})$ has the same weight $\vert \bar I\vert$.
Consequently, $T (\langle \Theta\rangle^{\bar I})$ is a $C^\infty(M)$-linear combination of basis elements $\langle \Theta\rangle^J$ where the multi-indices $J$ have weight $|J|=|\bar I| = Q-|I|$. Finally, $\star T (\Theta^{\bar I}) $ has weight $Q- |J|= |I|$.
We have therefore shown that the morphism $\star T\star$ maps elements of weight $|I|$ to elements of weight $|I|$, and so $\star T\star$ respects the gradation. 
Lemma \ref{lem_Phi} implies the rest of Part (1).

Part (2) follows from the definition of $D^t$ and its link with $\star$ covered in Subsection \ref{subsubsec_star}.
\end{proof}

\subsubsection{Properties of $\gr$ for algebraic maps}

Here, we describe some properties of $\gr$ in relation with algebraic maps (for the latter concept see Section \ref{subsubsec_def_alg}).

The following properties are easily checked.
\begin{lemma}
\label{lem_algebraic}
Let $M$ be a smooth manifold whose tangent bundle is filtered by vector subbundles as in \eqref{eq_filtredTM}. 
	\begin{enumerate}
		\item Any algebraic morphism of $C^\infty(M)$-module $D$ acting either on $\Omega^\bullet (M)$ or $\Gamma (\wedge^\bullet  \gr(T^*M))$
is determined by its pointwise restriction 
$D_x$ at each $x\in M$.
Conversely, any smooth map on $\wedge^\bullet  T^*M$ or $\wedge^\bullet  \gr(T^*M)$ that is linear at each fibre yields an algebraic morphism of $C^\infty(M)$-modules. \label{part 1, Lemma 4.19}

\item If a morphism of $C^\infty(M)$-modules $D$ acting on $\Omega^\bullet (M)$ 
is algebraic and respects the filtration, then $\gr( D)$ is a morphism of $C^\infty(M)$-module acting on $\Gamma (\wedge^\bullet  \gr(T^*M))$ which is also an algebraic map that respects the gradation.
Moreover, if $D$ is invertible, then
 $\gr( D)$ is also invertible, and $D^{-1}$ and $\gr(D)^{-1}$ are algebraic. \label{part 2, Lemma 4.19}
	\end{enumerate}
\end{lemma}

The isomorphism $\Phi$  together with the concept of algebraic maps  allow us to construct inverses in some cases, and this provides a partial converse to \eqref{part 2, Lemma 4.19} in Lemma \ref{lem_algebraic} as follows.

\begin{proposition}
	\label{prop_invert}
	Let $M$ be a smooth manifold whose tangent bundle is filtered by vector subbundles as in \eqref{eq_filtredTM}. 
	Assume that $M$ is equipped with a metric $g_{TM}$. We consider the corresponding map $\Phi$ defined  in Section \ref{subsubsec_PhiM}.
	
Let $D:\Omega^\bullet (M) \to \Omega^\bullet (M)$ be a morphism of $C^\infty(M)$-modules that respects the filtration and such that  $\gr (D)$ is algebraic.
Assume also that $\gr( D)$ is fiberwise invertible in the sense that its restriction $\gr( D)_x$ at each $x\in M$ is a linear invertible map on $\wedge^\bullet T_x^* M$.
Then the map $D$  is invertible
and the map $D^{-1} :\Omega^\bullet (M) \to \Omega^\bullet (M)$ respects the filtration.
Moreover,  $\gr (D^{-1})=  \gr( D)^{-1}$,
and 
if $D$ is in addition a differential operator acting on $\Omega^\bullet (M)$, then so is its inverse. 
\end{proposition}

\begin{proof}
By \eqref{part 2, lemma 4.16} in Lemma \ref{lem_Phi}, the map $T:=  (\gr (D))^ \Phi\colon\Omega^\bullet(M)\to\Omega^\bullet(M)$ respects the filtration and $\gr( T)^{-1} =\gr( D)^{-1}$. Moreover, 
by \eqref{part 2, Lemma 4.19} of Lemma \ref{lem_algebraic}, 
$T$ is invertible and its inverse is 
$T^{-1} = ( \gr( D)^{-1})^\Phi$. 

By Lemma \ref{lem_gr_modulemorphism}, 
we have
$$
\gr (T^{-1}D) =\gr (T^{-1}) \gr (D) = \gr( D)^{-1} \gr (D) = \id\ , 
\mbox{ so }
\gr (T^{-1}D -\id ) =\gr (T^{-1}D) -\id =0\ .
$$
Hence $T^{-1}D-\id =: B$ is nilpotent with $B^{N_0}=0$ by \eqref{eq_N_0}.
Since $T^{-1}D=\id+B$ and $\id +B$ is invertible with $(\id+B)^{-1} = \sum_{j=0}^{N_0} (-1)^j B^j$, 
we have $(\id+B)^{-1} T^{-1} D =\id$. 
This provides a left inverse for $D$ and the formula below with $B=B_L$. We can do the same on the right. 
\end{proof}

\begin{corollary}
	\label{cor_prop_invert}
	We continue with the setting of Proposition \ref{prop_invert} and its proof.
The following formulae holds for $D^{-1}:\Omega^\bullet (M) \to \Omega^\bullet (M)$:
$$
D^{-1} = \sum_{j=0}^{N_0}(-1)^j B_L^j T^{-1} = T^{-1}\sum_{j=0}^{N_0}(-1)^j B_R^j 
$$
where $T =\gr( D)^ \Phi = \Phi \gr( D)^{-1} \Phi^{-1}$, 
and $B_L = T^{-1}D-\id$ while $B_R = DT^{-1}-\id$.
\end{corollary}

\section{Filtered manifolds and osculating objects}
\label{sec_filtM+osculating}

In this section, we discuss the setting of filtered manifolds and some complexes naturally associated with them. 
In particular, we will define the osculating differential $d_{\fg M}$ and study its relation to the de Rham complex $(\Omega^\bullet(M),d)$.
We will also use the algebraic part $\dalg$ of $d$.
The main result of this section is in Proposition \ref{prop_grd} stating that $\gr (d)=\gr (\dalg)=d_{\fg M}$.

\subsection{Setting and definitions}
\label{subsec_def+setting}

A \emph{filtered manifold} is a smooth manifold $M$ equipped with a filtration of the tangent bundle $TM$ by vector subbundles 
$$
M\times \{0\} = H^0 \subseteq H^1 \subseteq \ldots \subseteq H^s =TM
$$
 satisfying
 \begin{equation}
 	\label{eq_GammaHiji+j}
 	[\Gamma(H^i),\Gamma(H^j) ]\subseteq \Gamma(H^{i+j}) \,,
 \end{equation}
with the convention that $H^{i} =TM$ when $i>s$.

\begin{ex}
    Any contact manifold, or more generally any regular subRiemannian manifold, is a filtered manifold, see Section \ref{subsubsec_gfgM}.
    A related class of examples is given by nilpotent Lie groups equipped with a left-invariant filtration \cite{LD+T22}.
\end{ex}

\subsubsection{Bundle of osculating Lie groups and algebras}

For each $x\in M$, the quotient
$\gr(T_x M) = \oplus_{i=1}^{s_0} \left(H^i_x / H^{i-1}_x\right)$
is naturally equipped with a Lie bracket $[\cdot,\cdot ]_{\fg_x M}$, since
$$
\forall\, f,g\in C^\infty(M), \ X\in \Gamma(H^i), \ Y\in \Gamma (H^j)\qquad
[fX,gY] = 
fg [X,Y] + \Gamma (H^{i+j-1}).
$$
When $\gr(T_x M)$ is equipped with this Lie bracket, 
we denote the resulting Lie algebra as
$$
\fg_x M := (\gr(T_x M), [\cdot,\cdot ]_{\fg_x M})\,.
$$
It is naturally graded by
$$
\fg_x M = \oplus_{i=1}^{s_0} \left(H^i_x / H^{i-1}_x\right),
$$
and is therefore nilpotent. 
We denote by $G_x M$ the corresponding connected simply connected nilpotent  Lie group (sometimes called the nilpotentisation or the tangent cone \cite{Mitchell} of $M$ at $x$). 
The unions
$$
G M :=\cup_{x\in M} G_x M
\qquad \mbox{and}  \qquad
\fg M :=\cup_{x\in M}\fg_x M,
$$
are naturally equipped with a smooth bundle structure that are called
\cite{RotschildStein,vErpYuncken}
the {\it bundles of osculating Lie groups and Lie algebras} over $M$.

We observe that the notions of weights defined in Section \ref{subsubsec_wedgekgrTM*}
and for the graded Lie algebra $\fg_x M$  coincide. 
We therefore obtain an analogous decomposition of $ \wedge^\bullet \fg_x^* M$ in terms of weights and degrees:
 \begin{equation}
 \label{eq_decwedgefgxM*}
 \wedge^\bullet \fg_x^* M = \oplus_{k,w\in \bN_0} \wedge^{k,w} \fg_x^* M \ , 
 \quad\forall\, x\in M\ .
 \end{equation}
 
We introduce the following vocabulary.
\begin{definition}
The elements of $G\Omega^\bullet (M) :=\Gamma (\wedge^\bullet\mathfrak{g}^\ast M)$ are  called \emph{osculating forms}. Moreover, for any $k,w\in \bN_0$, 
$$
G\Omega^k (M) :=\Gamma (\wedge^k \fg^* M)
\qquad\mbox{and}\qquad 
G\Omega^{k,w} (M) :=\Gamma (\wedge^{k,w} \fg^* M)
$$
are  called \emph{osculating $k$-forms} (or osculating forms of degree $k$) and 
\emph{osculating $k$-forms of weights $w$} respectively. 
\end{definition}

We say that
a map $T\colon G\Omega^\bullet (M)\to G\Omega^\bullet (M)$ \textit{respects the weights} of osculating forms when $T (G\Omega^{\bullet ,w}(M) )\subseteq G\Omega^{\bullet ,w}(M) $, and it \textit{increases the weights} when $T (G\Omega^{\bullet ,w} (M))\subseteq \oplus_{w'\geq w} G\Omega^{\bullet ,w'} (M)$.

\subsubsection{The osculating Chevalley-Eilenberg differential}
\label{subsec_dfgM}
For each $x\in M$,  
$d_{\fg_x M}$ denotes the Chevalley-Eilenberg differential on the  Lie group $G_x M$
viewed as the map 
$$
d_{\fg_x M} : \wedge^\bullet \fg_x^* M \to \wedge^\bullet \fg_x^* M
$$
defined via $d_{\fg_x M} (\wedge^0 \fg_x^* M)=\{0\}$ for $k=0$, for $k=1$
\begin{equation}
\label{eqdef_dgMk=1}
\forall\, \alpha\in \wedge^1 \fg_x^* M\, , \ V_0,V_1\in \fg_x M\ ,\quad
d_{\fg_x M}\alpha(V_0,V_1) = -\alpha ([V_0,V_1]_{\fg_x M}),
\end{equation}
 and more generally for $k>0$, for any $\alpha \in \wedge^k \fg_x^* M $ and $V_0,\ldots V_k\in \fg_x M$,
\begin{equation}
\label{eqdef_dgM}
	d_{\fg_x M}  \alpha (V_0,\ldots,V_k) = 
\sum_{0\leq i<j\leq k}
(-1)^{i+j} \alpha([V_i,V_j]_{\fg_x M},V_0,\ldots,\hat V_i,\ldots,\hat V_j,\ldots, V_k).
\end{equation} 
It is well known \cite{F+T1} that $d_{\fg_x M}$ is a linear map such that 
$$
d_{\fg_x M}^2 =0 \ , \quad d_{\fg_x M} (\wedge^{k} \fg_x^* M)\subseteq \wedge^{k+1} \fg_x^* M,
$$
and that it satisfies 
the Leibniz property on $\wedge^\bullet \fg_x^* M $, i.e.
$$
\forall\, \alpha\in \wedge^k \fg_x^* M , \ \beta\in \wedge^\bullet \fg_x^* M \ ,\quad
d_{\fg_x M} (\alpha\wedge \beta) = (d_{\fg_x M} \alpha)\wedge \beta + (-1)^k \alpha \wedge (d_{\fg_x M} \beta).
$$ 
Moreover,  $d_{\fg_x M}$ also preserves the weights of $\wedge^\bullet \fg_x^* M $:
$$
d_{\fg_x M} (\wedge^{k,w} \fg_x^* M )\subseteq \wedge^{k+1,w} \fg_x^* M.
$$

By construction, the algebraic map $d_{\fg M}$ defined by 
$$
d_{\fg M} : G\Omega^\bullet (M)\to G\Omega^\bullet (M) 
\qquad (d_{\fg M})_x = d_{\fg_x M}, \ \forall\,x\in M,
$$
is smooth. 
We call $(G\Omega^\bullet(M),d_{\fg M})$ the \emph{osculating complex} or \emph{osculating Chevalley-Eilenberg differential}. 

\subsubsection{The osculating box operator}
\label{subsec_oscBox}
Here we assume that 
the vector bundle $\wedge^\bullet \fg^* M$ 
is equipped with a metric  $g_{\wedge^\bullet \fg^* M}$:
each $\wedge^\bullet \fg_x^* M$ is equipped 
with a scalar product $g_{\wedge^\bullet \fg_x ^*M}$ with a smooth dependence in $x\in M$.
This scalar product allows us to consider, at each point $x\in M$, the transpose of the osculating differential $d_{\fg M}$, and then to define 
$$
\Box_{\fg M}:=
d_{\fg M}d_{\fg M}^t+d_{\fg M}^td_{\fg M} \,.
$$
We call this operator
 the \emph{osculating box}. 
 As $d_{\fg_M}$ is algebraic, so are $d_{\fg M}^t$ and $\Box_{\fg M}$.
 
 Let us now assume that the decomposition \eqref{eq_decwedgefgxM*}
  is orthogonal for the scalar product $g_{\wedge^\bullet \fg_x^* M}$ for each $x\in M$.
 This arises naturally when the manifold is filtered and Riemannian (see Section \ref {subsec_filtration+metric}).
With this assumption, we readily check that $d_{\fg M}^t$ preserves the weights, i.e. $d_{\fg_x M}^t ( \wedge^{k,w} \fg_x^* M)\subseteq  \wedge^{k-1,w} \fg_x^* M$,
 and
 $\Box_{\fg M}$ respects the degree and weights of osculating forms:
 $$
\Box_{\fg_x M} (\wedge^{k,w} \fg_x^* M )\subseteq \wedge^{k,w} \fg_x^* M.
$$
Moreover, given the scalar product $g_{\wedge^\bullet \fg_x^* M}$, the linear map
$$
\Box_{\fg_x M}:=
(\Box_{\fg M})_x   = d_{\fg_x M}d_{\fg_x M}^t+d_{\fg_x M}^td_{\fg_x M}\colon \wedge^\bullet \fg_x^* M\to\wedge^\bullet \fg_x^* M
$$
is  symmetric.
We denote by $\Pi_{\fg_x M}$ the  spectral orthogonal projection onto its kernel 
$$
E_{\fg_x M}:=\ker( \Box_{\fg_x M}) = \ker d_{\fg_x M}\cap \ker d_{\fg_x M}^t = \IM \Pi_{\fg_x M}.
$$
Note that the image of $\Box_{\fg_x M}$ is 
$$
F_{\fg_x M}:=\IM( \Box_{\fg_x M}) = \IM d_{\fg_x M}+  \IM d_{\fg_x M}^t = \ker \Pi_{\fg_x M} = E_{\fg_x M} ^\perp.
$$
By \cite[Lemma 2.1]{F+T1}, 
the complex $(F_{\fg M}^\bullet, d_{\fg M})$ is acyclic, i.e.
$d_{\fg M} (F_{\fg M})\subseteq F_{\fg M}$,
and the resulting cohomology is trivial:
$$
	d_{\fg_x M} (F_{\fg_x M}) = \ker (d_{\fg_x M} : F_{\fg_x M} \to F_{\fg_x M}) \subseteq \wedge^\bullet \fg_x^* M.
$$

For $|z|=\eps$ small enough, 
$z-\Box_{\fg_x M}$ is invertible, and by the Cauchy residue formula we have
$$
\Pi_{\fg_x M} = \frac1{2\pi i}\oint_{|z|=\eps} (z-\Box_{\fg_x M})^{-1} dz ,
$$
where the contour integration is over a circle about 0 of radius $\eps>0$ small enough.
This defines a smooth algebraic map $\Pi_{\fg M}$ acting on $G\Omega^\bullet (M)$ that respects degrees and weights.

\subsubsection{The partial inverse $d^{-1}_{\fg M}$}
\label{subsubsec_partialinv}

Proceeding as in \cite{F+T1}, 
at each $x\in M$,
a partial inverse $d_{\fg_x M}^{-1}$ of $d_{\fg_x M}$ can be defined as
$$
d_{\fg_x M}^{-1} := (d_{\fg_x M})^{-1} \pr_{\IM d_{\fg_x M}}\,,
$$ 
where we write $\pr_S$ to denote the orthogonal projection onto any $S$ closed subspace of $\wedge^\bullet \fg_x^* M$.

Note that in our construction in Section \ref{sec_generalscheme}, we will not use $d_{\fg_x M}^{-1}$.
However, it is used in Rumin's construction, see Section \ref{subsec_eqDtildeD}.

 where we show that the two constructions coincide. 

The resulting map $d_{\fg M}^{-1}$ on $G\Omega (M)$ is a smooth map that respects the filtration, and such that
$$
(d_{\fg M}^{-1})^2 =0\ \text{ and }\,
d_{\fg_x M}^{-1} (\wedge^{k+1,w} \fg_x^* M) \subseteq \wedge^{k,w}\fg_x ^*M\ ,\ \forall\, k,w\in \bN_0\,.
$$
By \cite[Section 2]{F+T1},  the kernel and image of $d_{\fg M}^{-1}$ coincide with the one of $d_{\fg M}^{t}$:
$$
\ker d_{\fg_x M}^t = \ker d_{\fg_x M}^{-1}
\quad\mbox{and}\quad 
\IM d_{\fg_x M}^t = \IM d_{\fg_x M}^{-1}, 
$$
so 
$$
E_{\fg_x M} = \ker d_{\fg_x M}\cap \ker d_{\fg_x M}^{-1}
\quad\mbox{and}\quad 
F_{\fg_x M}= \IM d_{\fg_x M}+  \IM d_{\fg_x M}^{-1}.
$$
By \cite[Proposition 2.2]{F+T1}, 
$$
\Pi_{\fg M} = \id -d_{\fg M}^{-1} d_{\fg M} - d_{\fg M} d_{\fg M}^{-1}
$$
 and we have
$$ 	d_{\fg M}^{-1}\Pi_{\fg M}  = \Pi_{\fg M} d_{\fg M}^{-1}
 	 =0\ ,\quad
 	 d_{\fg M}\Pi_{\fg M} = \Pi_{\fg M}d_{\fg M}=0\,.
$$

\subsection{The maps $d$ and $\dalg$ on a  filtered manifold}
Let us show that $\gr ( d)$ and $\gr ( \dalg)$ are well-defined, crucial operators, and that both $(\Gamma(\wedge^\bullet\gr(T^\ast M)),\gr(d))$ and $(\Gamma(\wedge^\bullet\gr(T^\ast M)),\gr(\tilde{d}))$ coincide with the osculating complex.

\begin{proposition}
\label{prop_grd}
 The maps $d$ and $\dalg$ respect the filtration with $$
 (d-\dalg)\Omega^{k,\geq w}(M) \subseteq \Omega^{k,\geq w+1}(M),
 $$
 for any degree $k\in \bN_0$ and weight $w\in \bN_0$.
 Moreover, 
 the maps 
 $\gr (d)$ and $\gr (\dalg)$ acting on $G\Omega^\bullet (M)\cong\Gamma(\wedge^\bullet\gr(T^\ast M))$ 
 coincide with $ d_{\fg M}$:
 $$
 \gr ( d) = \gr (\dalg) = d_{\fg M}
\,.
$$
	 \end{proposition}

\begin{proof}
From the definitions of $d$ and $\tilde d$, both differentials as maps $\Omega^1(M) \to \Omega^2(M)$ respect the corresponding filtrations. 
As they obey the Leibniz rule, both $d$ and $\tilde d$  respect the filtration on $\Omega^\bullet(M)$.
	Moreover, $\gr (d)$ and $\gr ( \dalg)$ satisfy the Leibniz property on $G\Omega^\bullet (M)\cong\Gamma(\wedge^\bullet\gr(T^\ast M))$.
	
	Let $\omega\in \Omega^{k,\geq w}(M)$ and for each $j=0,\ldots,k$ take
$V_j\in H^{w_{i(j)}}$ with $w_{i(0)} + \ldots +w_{i(k)}\leq w$.
Since $w_{i(0)} + \ldots +\hat {w}_{i(l)}+ \ldots +w_{i(k)}< w$, we have that 
$\omega(V_0,\ldots,\hat V_l,\ldots, V_k)=0$, and so 
$(d-\dalg)\omega(V_0,\ldots,V_k)
=0$ by \eqref{eq_d-d0}. 

This shows $(d-\dalg)\Omega^{k,\geq w} (M)\subseteq \Omega^{k,\geq w+1}(M)$, and implies $\gr( d)=\gr( \dalg)$. 
We are left to prove that $\gr(\dalg) = d_{\fg M}$.

If $\omega\in  \Omega^1(M)$ and $V_{i_1}\in \Gamma(H^{w_{i(1)}}),V_{i_2}\in \Gamma(H^{w_{i(2)}})$, 
then by Lemma \ref{lem_d0} and \eqref{eq_GammaHiji+j}.
\begin{align*}
\dalg \omega (V_1 \, \mod \, H^{w_{i(1)}-1}, V_2 \,  \mod \, H^{w_{i(2)}-1})
&= -\omega ([V_1 \, \mod \, H^{w_{i(1)}-1}, V_2\, \mod \, H^{w_{i(2)}-1}])
\\&=-\omega ([V_1,V_2]\, \mod  \, H^{w_{i(1)}+w_{i(2)}-1})\,.	
\end{align*}
Hence, we have that for any $\omega\in (H^{w_j})^\perp$, $j=0,1,\ldots,s_0-1$
$$
\gr(\dalg) (\omega \, \mod\, (H^{w_{j+1}})^\perp) = -(\omega  \, \mod\, (H^{w_{j+1}})^\perp) ([\cdot,\cdot]_{\fg M}).
$$
We recognise $d_{\fg M} (\omega \, \mod\, (H^{w_{j+1}})^\perp)$ from \eqref{eqdef_dgMk=1}.
Hence  the maps $d_{\fg_M}$ and $\gr ( \dalg) $ coincide 
on $G\Omega^1(M)$.
Since they both vanish on 
$G\Omega^0(M)$ and satisfy the Leibniz property, 
they coincide in fact on the whole $G\Omega^\bullet(M)$. 
\end{proof}

\section{Construction of subcomplexes on  Riemannian filtered manifolds}
\label{sec_generalscheme}

Here, we present the general scheme to construct  subcomplexes on a filtered manifold $M$ whose tangent bundle is equipped with a metric $g_{TM}$.
It relies on the notions of a base differential and codifferential that we introduce in Section \ref{subsec_base} (this terminology is borrowed from \cite{grong2023filtered}).
In Section \ref{subsec_BoxPL}, we construct certain operators $\Box,P,L$  which allow us to define the subcomplexes $(F_0^\bullet,C)$ and $(E_0^\bullet,D)$, with the latter computing the same cohomology as de Rham's (see Section 
\ref{subsec_DC}).
In Section \ref{subsec_tildeDC}, we also present an analogous construction, obtaining the operators $\tilde \Box,\tilde P,\tilde L$ and then the subcomplexes $(\tilde F_0^\bullet,\tilde C)$ and $(\tilde E_0^\bullet,\tilde D)$, with the latter computing the same cohomology as the complex $(\Omega^\bullet(M),\tilde d)$.
We present an alternative construction for $(E_0^\bullet,D)$ and
$(\tilde E_0^\bullet,\tilde D)$
in Section \ref{subsec_eqDtildeD} that follows Rumin's ideas. 

\subsection{Base differential and codifferential $d_0$ and $\delta_0$}
\label{subsec_base}

\begin{definition}
\label{def_basepair}
   Let $d_0$ and $\delta_0$ be two differential algebraic complexes on a filtered manifold $M$ equipped with a metric $g_{TM}$.
    We say that $(d_0,\delta_0)$ is a \emph{pair of base differential and codifferential}  when the following properties  are satisfied:
    \begin{itemize}
        \item $d_0$ increases the degree by one while $\delta_0$ decreases the degree by one, and both  respect the filtration of $\Omega^\bullet (M)$:
        $$
        d_0(\Omega^{k,\geq w}(M)) \subset \Omega^{k+1,\geq w}
        (M)\quad\mbox{and}\quad 
        \delta_0(\Omega^{k+1,\geq w}(M) )\subset \Omega^{k,\geq w}(M),
        $$
        \item $d_0$ and $\delta_0$ are disjoint \cite{Kostant}, that is, for any $\alpha\in \Omega^\bullet(M)$
        \begin{align*}
            \delta_0d_0\alpha=0\Longrightarrow d_0\alpha=0\ \text{ and }\ d_0\delta_0\alpha=0\Longrightarrow \delta_0\alpha=0\,,
        \end{align*}
        \item $\gr (d_0) = d_{\fg M}$ and $\gr (\delta_0) = d_{\fg M}^t$.
    \end{itemize}
\end{definition}
Above, the transpose is defined using a metric $g_{\wedge^\bullet \fg^* M}$ on the osculating bundle $\wedge^\bullet \fg ^*M$ induced by $g_{TM}$.

Given a pair $(d_0,\delta_0)$ of base differential and codifferential on a filtered manifold $M$, we define the associated \emph{base box} via
$$
  \Box_0 := d_0 \, \delta_0 +\delta_0\, d_0 .
$$

For any $x\in M$, we define the operator 
$$
\Pi_{0,x} := \frac1{2\pi i}\oint_{|z|=\epsilon} (z-\Box_{0,x})^{-1} dz ,
$$
for $\eps>0$ small enough.

\begin{proposition}
   \begin{enumerate}
       \item The operators $\Box_0$ and $\Pi_0$ are smooth and algebraic on $\Omega^\bullet (M)$.
       They respect its filtration and keep the degrees of the forms constant:
       $$
        \Box_0(\Omega^{k,\geq w}(M)) \subset \Omega^{k,\geq w}
        (M)\quad\mbox{and}\quad 
        \Pi_0(\Omega^{k,\geq w}(M) )\subset \Omega^{k,\geq w}(M).
        $$
        They also  satisfy 
          $$
    \gr(\Box_0) =\Box_{\fg M}
    \qquad\mbox{and}\qquad 
    \gr (\Pi_0) =\Pi_{\fg M}.
    $$
    The algebraic subbundles 
$$
E_0:=\ker d_0\cap \ker \delta_0
\quad\mbox{and}\quad 
F_0:= \IM d_0+  \IM \delta_0,
$$
 admit subfiltrations and we have:  
    $$
    \gr (E_0) =E_{\fg M}
     \qquad\mbox{and}\qquad 
    \gr (F_0)=F_{\fg M}.
    $$
       \item  
               We have
        $$
        \Omega^\bullet (M) = E_0^\bullet \oplus F_0^\bullet
        \qquad\mbox{with}
\qquad     E_0 = \ker \Box_0 
        \qquad\mbox{and}\qquad 
    F_0 =\IM \Box_0.
    $$
Moreover, $\Pi_0$ is the projection onto $E_0$ along $F_0$.
   \end{enumerate} 
\end{proposition}
\begin{proof}
Part (1) is satisfied by construction.
Disjointedness implies readily that
$E_0 \cap F_0 = \{0\}$.
Since $E_{\fg M} \oplus F_{\fg M} =\wedge^\bullet \fg^* M$, 
 Part (1) implies that $\dim E_0 +\dim F_0 =\dim E_{\fg M} + \dim F_{\fg M} = \dim \Omega^\bullet (M)$. Hence $E_0 \oplus F_0 = \Omega^\bullet (M)$.

Disjointedness also implies 
$E_0 = \ker \Box_0$.
Clearly, we also have $\IM \Box_0 \subset F_0$.
From the Cauchy residue formula, proceeding  as in the proof of 
 \cite[Proposition 4.1]{F+T1}, we check that $\Pi_0^2=\Pi_0$.
 In other words,  $\Pi_0$ is a  projection of $\Omega^\bullet (M)$.
 We also check using the Cauchy residue that $\Pi_0 =\id$ on $\ker \Box_0$
 while $\Pi_0 \Box_0 = 0$ so $\Pi_0 = 0 $ on $\IM \Box_0$.
We obtain Part (2) from $\Omega^\bullet (M) = E_0 \oplus F_0$.
\end{proof}

We call $\Pi_0$ the \textit{base kernel projection}.

\subsection{The operators $\Box$, $P$ and $L$}
\label{subsec_BoxPL}

We  define the differential operator acting on $\Omega^\bullet (M)$:
$$
\Box := d \, \delta_0 +\delta_0 \, d .  
$$

\begin{remark}
\label{rem_insteadd}
  The de Rham differential $d$ will be replaced with its algebraic part $\tilde d$ in Section \ref{subsec_tildeDC}.
    The new construction will closely follow the steps below, as the only property that it relies on is that $\gr(d)=d_{\fg M}$. 
\end{remark}

\begin{proposition}
\label{prop_Box}
The operator $\Box$ commutes with $d$ and $\delta_0$. 
It is a  differential operator that preserves the degrees and the weights of  forms:
 $$
\Box (\Omega^{k,\geq w} (M))\subseteq \Omega^{k,\geq w} (M).
$$
We have
$$
\gr(\Box)  =\Box_{\fg M}.
$$
\end{proposition}
\begin{proof}
Since $d$ and $\delta_0$ are complexes, they commute with $\Box$.
Moreover, $d$ and $d_0$ increase the degrees of forms by one, while $\delta_0$ decreases the degrees by one. Hence, $\Box$ and $\Box_0$ preserve the degrees of forms. 
As $d$, $d_0$ and $\delta_0$  respect the filtration, so do $\Box$ and $\Box_0$. 
We conclude with 
$$
\gr (\Box) = 
\gr (d) \gr(\delta_0)+\gr(\delta_0)\gr (d)
=d_{\fg M} d_{\fg M}^t+ d_{\fg M}^td_{\fg M} = \Box_{\fg M}.
$$
\end{proof}

Applying Proposition \ref{prop_Box} together with Proposition \ref{prop_invert} to   $z-\Box$,
 the map 
$$
P := \frac1{2\pi i}\oint_{|z|=\eps} (z-\Box)^{-1} dz ,
$$
is well-defined for  $|z|=\eps$ small enough. 

\begin{proposition}
\label{prop_P}
	\begin{enumerate}
		\item The map $P$ is a differential operator acting on $\Omega^\bullet (M)$ that respects the filtration and the degree of the forms with $\gr ( P)  =\Pi_{\fg M}$.
It commutes with $\Box$, $d$ and $\delta_0$.
			\item 
We have 
$\Box^{N_0}P =0$ while $\Box^{N_0}$ acts on $\IM P$ where it is invertible. 
Moreover,		
$P$ is the projection onto
$$
E:=\ker \delta_0  \cap \ker (\delta_0  d) = \IM P= \ker \Box^{N_0},
$$
along
$$
F:= \IM \delta_0  + \IM d\delta_0=\ker P = \IM \Box^{N_0}.
$$
These modules satisfy $\gr (E) =E_0$ and $\gr (F) = F_0$. 
			\end{enumerate}
\end{proposition}
\begin{proof}
Part (1) follows from 
Proposition \ref{prop_Box} and  applying $\gr$.
 From the Cauchy residue formula, proceeding  as in the proof of 
 \cite[Proposition 4.1]{F+T1}, we check that $P^2=P$.
 In other words,  $P\colon\Omega^\bullet (M)\to\Omega^\bullet(M)$ is a  projection.
				
For Part (2), $\gr (\Box P )= 0$ by the Cauchy formula and the properties of $\gr$. 
Therefore $(\Box P )^{N_0} =0$, but $(\Box P )^{N_0} = \Box ^{N_0} P $ since $P $ is a projection commuting with $\Box $. 
Therefore $\IM P  \subseteq \ker \Box ^{N_0}$.

Naturally, $\Box $ acts on $\IM \Box ^{N_0}$ which is a submodule of $\Omega^\bullet (M)$ that inherits a filtration. Moreover, we check 
$$
\gr ( \IM \Box ^{N_0} )=  \IM  \gr (\Box ^{N_0}  )
=\IM \Box_{\fg M} = \ker \Pi_{\fg M}, 
$$
since $\Box_{\fg M}$ is symmetric and the orthogonal projection onto its kernel is $\Pi_{\fg M}$,
and
$$
\gr|_{\IM \Box ^{N_0}} (\Box :\IM \Box ^{N_0}\to \IM \Box ^{N_0}) = \Box_{\fg M}:\IM \Pi_{\fg M}\to \IM \Pi_{\fg M}.
$$ 
Adapting the proof of Proposition \ref{prop_invert} to a submodule, we obtain that $\Box |_{\IM \Box ^{N_0}}$ is invertible on $\IM \Box ^{N_0}$.
This implies that $\ker \Box ^{N_0}\cap \IM \Box ^{N_0}=\{0\}$ and $\ker \Box ^{N_0}\oplus  \IM \Box ^{N_0}=\Omega ^\bullet (M)$.
We then readily check that $P$ is the projection onto $\ker \Box^{N_0}$ along $\IM \Box^{N_0}$.

As $\Pi_{\fg M} d_{\fg M}^t=0$, 
we have $\Pi_0 \delta_0=0$.
Hence, since $P$ commutes with $\delta_0$, we have
$$
(P- \Pi_0 )\delta_0=P\delta_0  = \delta_0 P\,,
$$
and so recursively, for any $k\ge 1$,
$$
(P- \Pi_0)^{k} \delta_0 =P^{k}\delta_0= \delta_0 P^k\,.
$$
For $k\ge N_0$, $(P- \Pi_0)^{k}=0$
since $\gr(  P) = \Pi_{\fg M} = \gr (\Pi_0)$.
As $P$ is a projection, we have obtained 
$0 =P\delta_0 = \delta_0  P$. Moreover, since $P$ commutes with $d$, we also have $Pd\delta_0  = dP \delta_0  =0$, which implies the inclusion $\ker P \supset \IM \delta_0  + \IM d\delta_0 =:F$.

Since $d^2=0$ and $(\delta_0 )^2=0$, we compute easily 
$$
\Box^{N_0} = (d \delta_0 )^{N_0} +(\delta_0  d)^{N_0},
$$
As $P$ is the projection onto $\ker \Box^{N_0}$ along $\IM \Box^{N_0}$,  we have 
\begin{align*}
\IM P &=\ker \Box^{N_0}  \supset \ker \delta_0  \cap \ker (\delta_0  d) =: E,\\
\ker P &=\IM \Box^{N_0}  \subset \IM \delta_0  + \IM (d \delta_0 )=: F\,.	
\end{align*}
These last inclusions then imply that $\ker P= F$, but also $\IM P =E$.  
\end{proof}

Note that in our construction below, we do not need the characterisations of $\IM P$  and $\ker P$ as 
$
E=\ker \delta_0  \cap \ker (\delta_0  d)$
and $
F= \IM \delta_0  + \IM d\delta_0$.
We will only need it when proving that this construction coincides with Rumin's. 

We follow the  same strategy as in the case of homogeneous groups \cite[Section 4.1]{F+T1}:

\begin{proposition}
\label{prop_L}
	The differential operator $L$ acting on $\Omega^\bullet (M)$ and defined as 
	$$
	L:=P  \Pi_0 +(\id-P)(\id- \Pi_0),
	$$
preserves the degrees of the forms and respects the filtration with $\gr( L)=\id$. It is invertible and its inverse is a differential operator acting
on $\Omega^\bullet (M)$. 
We have
$$
PL = P  \Pi_0 = L \Pi_0
$$
and this implies 
$$
P=L  \Pi_0 L^{-1}
\qquad\mbox{and}\qquad
 (L^{-1} \Box L)  \Pi_0 =  \Pi_0(L^{-1} \Box L).
 $$
Consequently, 
$$
F=
\ker P = L (\ker  \Pi_0) L^{-1} = L F_0 L^{-1}, 
\qquad
E=\IM P = L (\IM  \Pi_0) L^{-1} = L E_0 L^{-1}.
$$
\end{proposition}

\begin{proof}
By Lemma \ref{lem_gr_modulemorphism}, we have
	$$
	\gr ( L) = \gr( P) \ \gr(  \Pi_0) + \gr (\id-P)\ \gr (\id- \Pi_0)
	=\Pi_{\fg M}^2 + (\id-\Pi_{\fg M})^2= \id.
	$$
	We conclude with Proposition \ref{prop_invert} and straightforward computations. 
\end{proof}

\subsection{The complexes $(E_0^\bullet,D)$ and $(F_0^\bullet,C)$}
\label{subsec_DC}
By Propositions \ref{prop_P} and \ref{prop_L},   
the de Rham differential	 $d$ commutes with  $P$, so $L^{-1} dL $ commutes with $L^{-1}PL = \Pi_0$. 
Hence,
we can decompose the differential operator $L^{-1} d L$ as
$$
L^{-1} d L \ = \ L^{-1} d L \Pi_0 \ + \ L^{-1} d L (\id -\Pi_0)
\ = \ D \ + \ C\,,
$$
where $D$ and $C$ are the differential operators
\begin{align*}
   D&:= L^{-1} d L \Pi_0 \ =\ \Pi_0 L^{-1} d L \Pi_0, \\
   C&:= L^{-1} d L (\id -\Pi_0)\ =\ (\id-\Pi_0) L^{-1} d L (\id -\Pi_0).
\end{align*}
Since 
$(L^{-1}dL)^2 =L^{-1}d^2 L=0$ and $DC=0=CD$, 
it follows that 
$D^2=C^2=0$.
Hence,
the chain complex $L^{-1}dL$ decomposes into the direct sum of the two chain complexes: $D$ acting on $E_0=\IM \Pi_0$, and $C$ acting on $F_0=\ker\Pi_0$.

\begin{proposition}
\label{prop_cohomF0}
		The maps $C$ and $d_0$ are conjugated on $F_0$:
$$
Cg=gd_0 \quad\mbox{on}\  F_0\,, 
$$
 where $g:F_0 \to F_0$ is the invertible operator defined as
$$
g:= C\delta_0 \Box_0^{-1} + \delta_0 \Box_0^{-1} d_0 \colon F_0\to F_0\,.
$$
Hence, the complexes  $(F_0^\bullet,C)$  and $(F_0^\bullet,d_0)$ have the same cohomology.	
\end{proposition}

\begin{proof}
Note that $\IM d_0 \subseteq \IM \Box_0 = F_0$ and that $\Box_0$ is invertible on $\IM \Box_0 = F_0$, so $g$ is well-defined on $F_0$. 

Since $\gr ( L) =\id$, $\gr( d)=d_{\fg_M}$ and $\gr ( \Pi_0) =\Pi_{\fg M}$, we have
$$
\gr( C)= \gr (L^{-1} d L(\id - \Pi_0))=
d_{\fg_M} (\id -\Pi_{\fg M}),
$$
and on $F_{\fg M}$
\begin{align*}
\gr|_{F_0}( g) 
&= \gr( C) \,  d_{\fg_M}^t \Box_{\fg M}^{-1} + d_{\fg_M}^t\Box_{\fg M}^{-1}d_{\fg_M}	\\
&=  d_{\fg_M} d_{\fg_M}^t \Box_{\fg M}^{-1} + d_{\fg_M}^t\Box_{\fg M}^{-1}d_{\fg_M}	\\
&=   (d_{\fg_M}d_{\fg_M}^t + d_{\fg_M}^td_{\fg_M})\Box_{\fg M}^{-1} = \id_{F_{\fg M}}\,.	
\end{align*}
Applying the proof of Proposition \ref{prop_invert} to the submodule $F_0$, $g$ is invertible on $\IM P_0=F_0$. 

Since $C^2=0$ and $d_0^2=0$, it is a straightfoward to check that
$Cg=gd_0$ holds on $F_0$.
\end{proof}

\begin{proposition}
The differential operator $ \Pi_0 L^{-1}$ 
is a chain map between $(\Omega^\bullet (M),d)$ and $(E_0,D)$, that is $( \Pi_0 L^{-1})d = D( \Pi_0 L^{-1})$.
This chain map is homotopically invertible, with homotopic inverse given by $L$
since we have
$$
( \Pi_0 L^{-1}) L = \Pi_0,
\quad\mbox{and}\quad
\id - L( \Pi_0 L^{-1}) = d h +hd\,,
$$
where $h$ is the differential operator acting on $\Omega^\bullet (M)$ and defined as
$$
h:=L\, g  \delta_0\Box_0^{-1} \, g^{-1}(\id - \Pi_0)\, L^{-1},
$$
with $g$ as in Proposition \ref{prop_cohomF0}.
\end{proposition}

\begin{proof}
	By construction, we have
	$d = L(D+C)L^{-1}$, with $D= \Pi_0 D \Pi_0$ and $C=(\id- \Pi_0) C (\id- \Pi_0)$,
	so 
	$ \Pi_0 L^{-1} d =  \Pi_0 D  \Pi_0 L^{-1}= D  \Pi_0 L^{-1}$.
	
	It remains to prove the properties regarding $h$.
	We first point out  that $h$ makes sense because $\IM \delta_0 \subset F_0$ 
and $g$ acts on $F_0 = \IM (\id -\Pi_0)$ in an invertible way.
The definitions of $h$ and $C$ together with Proposition \ref{prop_cohomF0} then 
yield:	
\begin{align*}
	L^{-1} (dh+hd) L
	&=
	L^{-1}d L\, g  \delta_0\Box_0^{-1} \, g^{-1}(\id - \Pi_0) +g \delta_0\Box_0^{-1} \, g^{-1}(\id - \Pi_0)\, L^{-1} d L
	\\
	&=
	C\, g  \delta_0\Box_0^{-1} \, g^{-1}(\id - \Pi_0) +g \delta_0\Box_0^{-1} \, g^{-1} C
	\\
	&=
	 gd_0  \delta_0\Box_0^{-1} \, g^{-1}(\id - \Pi_0) +g  \delta_0\Box_0^{-1} \,  d_0 g^{-1}(\id - \Pi_0)
		\\
	&= g\left( d_0\delta_0 \Box_0^{-1} + \delta_0 \Box_0^{-1}d_0\right) g^{-1} (\id- \Pi_0)
	= 	\id- \Pi_0.
	\end{align*}
	The conclusion follows. 
\end{proof}

\begin{corollary}
The cohomology of $(E_0^\bullet,D)$ is linearly isomorphic to the de Rham cohomology of the manifold $M$.
\end{corollary}

\subsection{The  complexes $( E_0^\bullet,\tilde D)$ and $(F_0^\bullet,\tilde C)$}
\label{subsec_tildeDC}

In this section, 
we consider the complexes obtained by considering $\tilde d$ instead of $d$ in the construction above (see Remark \ref{rem_insteadd}). 
The proofs are omitted, as the arguments are essentially the same as above. 
We start by  defining the algebraic $\tilde\Box$ operator acting on $\Omega^\bullet (M)$:
$$
\tilde \Box := \tilde d \, \delta_0 +\delta_0 \, \tilde d .  
$$
It commutes with $\tilde d$ and $\delta_0$ and preserves the degrees and the weights of  forms:
 $$
\tilde \Box (\Omega^{k,\geq w} (M))\subseteq \Omega^{k,\geq w} (M).
$$
We have
$$
\gr(\tilde \Box)  =\Box_{\fg M}.
$$
Applying Proposition \ref{prop_invert} to   $z-\tilde \Box$  locally at $x\in M$, together with Proposition \ref{prop_Box},
 the map 
$$
\tilde P := \frac1{2\pi i}\oint_{|z|=\eps} (z-\tilde \Box)^{-1} dz ,
$$
is well-defined for  $|z|=\eps$ small enough. $\tilde P$ is then an algebraic operator acting on $\Omega^\bullet (M)$ that respects the filtration and the degree of forms, and $\gr (\tilde P)  =\Pi_{\fg M}$.
It commutes with $\tilde \Box$, $\tilde d$ and $\delta_0$.
We have 
$\tilde \Box^{N_0}\tilde P =0$, while $\tilde \Box^{N_0}$ acts on $\IM \tilde P$ where it is invertible. 
Moreover,		
$\tilde P$ is the projection onto
$$
\tilde E:=\ker \delta_0  \cap \ker (\delta_0  \tilde d) = \IM \tilde P= \ker \tilde \Box^{N_0},
$$
along
$$
\tilde F:= \IM \delta_0  + \IM \tilde d\delta_0=\ker\tilde P = \IM \tilde \Box^{N_0}.
$$
These modules satisfy $\gr ( \tilde E) =E_0$ and $\gr (\tilde F) = F_0$. 

	The algebraic operator $\tilde L$  acting on $\Omega^\bullet (M)$ and defined as 
	$$
	\tilde L:=\tilde P  \Pi_0 +(\id-\tilde P)(\id- \Pi_0),
	$$
preserves the degree of forms, respects the filtration, and $\gr( \tilde L)=\id$. It is invertible and its inverse is an algebraic operator acting
on $\Omega^\bullet (M)$. 
We have
$$
\tilde P\tilde L = \tilde P  \Pi_0 = \tilde L \Pi_0, 
\qquad 
\tilde P=\tilde L  \Pi_0 L^{-1}
\qquad\mbox{and}\qquad
 (\tilde L^{-1} \tilde \Box \tilde L)  \Pi_0 =  \Pi_0(\tilde L^{-1} \tilde \Box \tilde L).
 $$
Consequently, 
$$
\tilde F=
\ker \tilde P = \tilde L (\ker  \Pi_0) \tilde L^{-1} = \tilde L F_0 \tilde L^{-1}, 
\qquad
\tilde E=\IM \tilde P = \tilde L (\IM  \Pi_0) \tilde L^{-1} = \tilde L E_0 \tilde L^{-1}.
$$

We define the two differentials $\tilde D$ on $E_0=\IM  \Pi_0$, and $\tilde C$ on $F_0=\ker  \Pi_0$ via
$$
\tilde L^{-1} \tilde d \tilde L   = \tilde L^{-1} \tilde d \tilde L  \Pi_0 +\tilde L^{-1} \tilde d \tilde L(\id - \Pi_0)=: \tilde D+\tilde C
$$

The maps $\tilde C$ and $d_0$ are conjugated on $F_0$:
$$
\tilde C\tilde g=\tilde gd_0 \quad\mbox{on}\  F_0\,, 
$$
 where $\tilde g:F_0 \to F_0$ is the invertible operator defined as
$$
\tilde g:= \tilde C\delta_0 \Box_0^{-1} + \delta_0 \Box_0^{-1} d_0\colon F_0\to F_0\,.
$$
 The complexes $(F_0^\bullet,\tilde C)$ and $(F_0^\bullet,d_0)$ have the same  cohomology.	

The differential operator $ \Pi_0 \tilde L^{-1}$ 
is a chain map between $(\Omega^\bullet (M),\tilde d)$ and $(E_0,\tilde D)$, that is $( \Pi_0 \tilde L^{-1})\tilde d = \tilde D( \Pi_0 \tilde L^{-1})$.
This chain map is homotopically invertible, with homotopic inverse given by $\tilde L$
since we have
$$
( \Pi_0 \tilde L^{-1}) \tilde L = \Pi_0,
\quad\mbox{and}\quad
\id -\tilde L( \Pi_0 \tilde L^{-1}) = \tilde d \tilde h +\tilde h \tilde d\,,
$$
where $\tilde h$ is the algebraic operator acting on $\Omega^\bullet (M)$ and defined as
$$
\tilde h:=\tilde L\, \tilde g  \delta_0\Box_0^{-1} \, \tilde g^{-1}(\id - \Pi_0)\, \tilde L^{-1}.
$$
Consequently, the cohomology of $(E_0^\bullet,\tilde D)$ is linearly isomorphic to the cohomology of $\tilde d$ of the manifold $M$.

\subsection{Equivalent constructions for the subcomplexes $(E_0^\bullet,D)$ and $(E_0^\bullet,\tilde D)$}
\label{subsec_eqDtildeD}
Here, we present an interpretation of Rumin's construction \cite{Rumin1999,RuminPalermo} of the complex that bears his name.
In our presentation, we highlight the difference between objects living on the manifold and their osculating counterparts.

We first define 
 the map 
\begin{equation}
	\label{eq_definvd0}
	d_0^{-1}
	 := \Phi^{-1} \circ d_{\fg M}^{-1}\circ\Phi\colon\Omega^\bullet(M)\to\Omega^\bullet(M) \,,
\end{equation}
which is defined via the partial inverse $d_{\fg M}^{-1}$ of the osculating differential $d_{\fg M}$ (see Section \ref{subsubsec_partialinv}).
From the properties of $d_{\fg M}^{-1}$, 
we see that 
 $(d_0^{-1})^2 =0$, it respects the filtration and decreases the degree by one:
$$
\forall\, k,w\in \bN_0\qquad
d_0^{-1} (\Omega^{k+1,\geq w} (M) )\subseteq \Omega^{k,\geq w}( M)\,.
$$
We have
$$
\ker d_0^t = \ker d_0^{-1}
\quad\mbox{and}\quad 
\IM d_0^t = \IM d_0^{-1}, 
$$
so 
$$
E_0 = \ker d_0\cap \ker d_0^{-1}
\quad\mbox{and}\quad 
F_0= \IM d_0+  \IM d_0^{-1}.
$$
Moreover, 
\begin{equation}
	\label{eq_Pi0wd0inv}
	\Pi_0 = \id -d_0^{-1} d_0 - d_0 d_0^{-1},
\end{equation}
 and we have
$$
d_0^{-1}\Pi_0  = \Pi_0 d_0^{-1}
 	 =0\ ,\quad
 	 d_0\Pi_0 = \Pi_0 d_0=0\,.
$$

	\begin{lemma}
	\label{lem_Pi} 
 	Let us consider the differential operators acting on $\Omega^\bullet (M)$ defined by:
 $$
 b:=d_0^{-1}d_0 - d_0^{-1} d=- d_0^{-1}(d-d_0)
 \quad\mbox{and}\quad
b_1:=d_0 d_0^{-1} -d d_0^{-1}= - (d-d_0)d_0^{-1}.
 $$
\begin{enumerate}
\item The maps $b$ and $b_1 $ are  nilpotent. Consequently,  $\id-b$ and $\id-b_1 $ are invertible, and $(\id-b)^{-1}$ and $(\id-b_1 )^{-1}$ are  well-defined differential operators acting on $\Omega^\bullet (M)$.
We have:
$$
bd_0^{-1}= d_0^{-1}b_1 , 
\quad\mbox{and}\quad 
(\id- b)^{-1}d_0^{-1} =d_0^{-1} (\id-b_1 )^{-1}.
$$
\item 
 The differential operator $\Pi\colon\Omega^\bullet(M)\to\Omega^\bullet(M)$ defined as
		 \begin{align*}
	\Pi&:=(\id-b)^{-1} d_0^{-1}d+d(\id-b)^{-1} d_0^{-1}= d_0^{-1}(\id-b_1 )^{-1} d +d d_0^{-1}(\id-b_1 )^{-1}\,,
	\end{align*}
	 is the projection onto
	$$
	 F:=\IM d_0^{-1} +  \IM dd_0^{-1}=\IM d_0^{t} +  \IM dd_0^{t}\,,
	 $$
	 along 
$$
 E:=\ker  d_0^{-1} \cap  \ker (d_0^{-1}d)
= \ker  d_0^{t} \cap  \ker (d_0^{t}d)\,.
$$
\end{enumerate}
\end{lemma}

\begin{proof}
By Proposition \ref{prop_grd} and Lemma \ref{lem_Phi},  $d$ and $d_0$ increase the degree of forms by one, respect the filtration, and $\gr(d)  = d_{\fg M}=\gr  (d_0)$.
As $d_0^{-1}$ decreases the degree of forms by one,    $b$ respects the filtration, and
$
\gr	(b)=- \gr (d_0^{-1}) \ \gr (d-d_0) =0$.
Consequently, it is nilpotent with $b^{N_0}=0$.
The proof then follows as in \cite[Lemma 3.9]{F+T1}.
\end{proof}

\begin{corollary}
	\label{cor_lem_Pi}
	The projections $\Pi$ and $P$ constructed in Lemma \ref{lem_Pi} and Proposition \ref{prop_P} respectively coincide.  
\end{corollary}

Following Rumin's notation, the two projections onto $F$ and $E$ 
along $E$ and $F$ are denoted respectively as
$$
\Pi_F:=\Pi 
\qquad\mbox{and}\qquad 
\Pi_E:=\id -\Pi\, .
$$
We then obtain Rumin's construction and its equivalence with the one presented in this paper.

\begin{theorem}
\label{thm_Rumin}

\begin{enumerate}
    \item (M. Rumin)
The de Rham complex $(\Omega^\bullet(M), d)$ splits into two subcomplexes $(E^\bullet,d)$ and $(F^\bullet,d)$. 
Moreover,   the differential operator defined as 
$$
d_c :=\Pi_0 d\, \Pi_E\, \Pi_0\colon\Omega^\bullet(M)\to\Omega^\bullet(M)\ ,  
$$
satisfies
$$
d_c^2=0\quad ,\quad d_c (\Omega^k (M))\subset \Omega^{k+1} (M)\,.
$$
Moreover, we have
$$
\Pi_E = \Pi_E \Pi_0 \Pi_E
\quad\mbox{and}\quad
\Pi_0 \Pi_E \Pi_0 = \Pi_0.
$$
and so $E_0= \IM \Pi_0 = \Pi_0  E$, and 
the complex $(E^\bullet,d)$ is conjugated to  $(E_0^\bullet,d_c)$ via $\Pi_0$.
\item 
	The  maps $d_c$ and $D$ coincide, i.e. 
$d_c = D$.
\end{enumerate}
\end{theorem}

\begin{proof}
Adapting the proof of \cite[Lemma 3.10]{F+T1}, 
it is easy to check that the operators $d_0^{-1}$, $\Pi_E$ and $\Pi_F$ satisfy the following properties:
\begin{itemize}
\item[1.]	
$d_0^{-1}\Pi_E = \Pi_E d_0^{-1}=0$;
\item[2.]
$d\Pi_F = \Pi_F d$, and $d\Pi_E = \Pi_E d$;
\item[3.]
$\Pi_E  (\id-\Pi_0)  \Pi_E =0$;
\item[4.]
 on $\Omega^k (M)$, we have
$\Pi_F^t=(-1)^{k(n-k)}\star \Pi_F \star $ and $\Pi_E^t=(-1)^{k(n-k)}\star \Pi_E \star $.
\end{itemize}
 These readily imply  Part (1) and Part (2), after adapting  the proofs of Theorems 3.11 
and 4.9 in  \cite{F+T1} respectively.
\end{proof}

This equivalent construction also works if we replace $d$ with $\tilde d$, yielding a complex $(E_0^\bullet,\tilde d_c)$ which coincides with $\tilde D$.

\subsection{Case of regular subRiemannian manifolds}
\label{subsec_regsubR}

Regular subRiemannian manifolds are filtered manifolds (see below). 
If we equip such a  manifold with a Riemannian metric, 
then our construction applies and we obtain two complexes $(E_0^\bullet,D)$ and $(E_0^\bullet,\tilde D)$.
Here, we show that if the Riemannian metric  is compatible with the  subRiemannian structure as explained in Section \ref{subsubsec_compatibleRg} below, then $(E_0^\bullet,D)$
  coincides at least locally with what is now customarily referred to as the Rumin complex.
  We would like to stress that, when introducing what Rumin calls the ``Carnot complex of an $E_0$-regular $CC$-structure'' in \cite{RuminPalermo}, he only briefly mentions the potential impact that the choice of a Riemannian metric can have on the resulting subcomplex of $M$ because he considers only local or osculating objects. In particular, he does not address any compatibility conditions between the Riemannian and subRiemannian structures, although it seems to be an important ingredient in the construction. 
We hope that our approach will lead to a better understanding of the constructed subcomplexes.

\subsubsection{The osculating metric of a regular subRiemannian manifold}
\label{subsubsec_gfgM}

Recall that a subRiemannian manifold is a smooth manifold $M$ equipped with a bracket generating distribution $\cD\subset TM$ and with a metric $g_\cD$ on $\cD$.
Let $\Gamma^1=\Gamma(\cD)$ be the set of smooth sections of $\cD$, and 
$\Gamma^i:=[\Gamma^1, \Gamma^{i-1}]+\Gamma^{i-1}$, for $i>1.$
As $\cD$ is bracket generating, 
there exists $i$ such that $\Gamma^i = TM$, 
and we denote by $s$ the smallest such integer $i$.
If, for every $i=1, \ldots,s$, there exists a subbundle $H^i\subset TM$ for which $\Gamma^i=\Gamma(H^i)$ is the set of smooth sections on $H^i$, then $M$ is said to have an \textit{regular} subRiemannian structure.  
Clearly, the $H^i$s provide the structure of filtered manifold.

Consider a regular subRiemannian manifold $M$ as above. 
In this case,  the osculating Lie algebras are stratified 
\begin{align*}
    \fg_xM=\oplus_{i=1}^s\fg_{i,x}, \qquad \fg_{i,x} = H^i_x / H^{i+1}_x\ ,\ \forall\,x\in M\,,
\end{align*}
and the subspaces $\fg_{i,x}\subset\fg_xM$ are given by imposing $\mathfrak g_1:=\mathcal{D}$, and
$$ 
\fg_{2,x} = [\fg_{1,x},\fg_{1,x}]_{\fg_x M},\quad \fg_{3,x}=[\fg_{1,x},\fg_{2,x}]_{\fg_xM}, \ldots , \
\fg_{i,x} = [\fg_{1,x},[\fg_{i-1,x}]_{\fg_x M},\ldots 
$$
The metric $g_\cD$ on $\cD$  naturally induces
 a scalar product on each osculating Lie algebra fibre $\fg_x M$ \cite[p. 188]{Montgomerybk}. Let us briefly recall its construction.
 
At every point $x\in M$,
the map
$$(\fg_{1,x})^j=\fg_{1,x} \times \ldots \times \fg_{1,x}
 \to  \fg_{j,x}\ ,\ \, 
(V_1, \ldots,  V_j) \longmapsto  [V_1, [V_2,[\cdots, V_j]_{\fg_x M}]_{\fg_x M}\cdots]_{\fg_x M}\ , 
$$
is $j$-linear and surjective. Hence, the metric $g_\cD$ on 
$\fg_{1}=H^1=\cD$ induces a scalar product on $(\fg_{1,x})^j$ and then on $\fg_{j,x}$ (which is the image of the above map, and therefore inherits the scalar product from the orthogonal complement of the kernel).
Constructing this  for $j=2,\ldots,s$, 
we obtain a scalar product $g_{\fg_x M}$ on $\fg_x M$.
By construction, the dependence in $x$ is smooth, 
and $(g_{\fg M})(x) := g_{\fg_x M} $, $x\in M$, defines 
an element $g_{\fg M}\in \Gamma ({\rm Sym}(\fg M\otimes \fg M))$.
Therefore, $g_{\fg M}$ is a metric on the osculating bundle $\fg M$  of Lie algebras.
We call $g_{\fg M}$ the \emph{osculating metric}  induced by $g_\cD$.

\subsubsection{Compatible Riemannian metrics}
\label{subsubsec_compatibleRg}

We now assume that, in addition to being a regular subRiemannian manifold, $M$ is also equipped with a Riemannian metric, that is, with a metric $g_{TM}$ on its tangent bundle $TM$.
As already seen in Sections \ref{subsec_Euclgr} and \ref{subsubsec_natorthgrad}, $g_{TM}$ will also induce a metric $g_{\gr(TM)}$ on $\gr (TM) \cong \fg M$.
The two metrics $g_{\fg M}$ and $g_{\gr(TM)}$ will be different in general.

\begin{definition}
\label{def_compatible}
    A Riemannian metric $g_{TM}$ is \emph{compatible with the subRiemannian structure} of a regular subRiemannian manifold $M$ (or compatible for short) when the two metrics  $g_{\gr(TM)}$ and $g_{\fg M}$ on the osculating Lie algebra bundle $\fg M\cong \gr(TM)$ 
   coincide. 
\end{definition}

By construction, $g_{\fg M}$ coincides on $\cD = H^1 = \fg_1$ with $g_{\cD}$. 
Hence, a necessary condition for $g_{TM}$ to be compatible is for $g_{TM}$ to coincide with $g_{\cD}$ on $\cD$.
However, it is not sufficient generally. For example, let us consider the 6-dimensional free nilpotent Lie group of rank 3 and step 2, denoted by $N_{6,3,6}$ in \cite{LD+T22}. Keeping the notation of \cite{LD+T22}, if we take $\Theta=\lbrace\theta^1,\theta
^2,\theta^3,\theta^4,\theta^5,\theta^6\rbrace$ and $\hat\Theta=\lbrace\theta^1,\theta^2,\theta^3,\theta^4,\theta^4+\theta^5,\theta^5+\theta^6\rbrace$ as global left-invariant orthogonal coframes of $TN_{6,3,6}$, then the corresponding subcomplexes $(E_0^\bullet,D)$ and $(\hat{E}_0^\bullet,\hat{D})$ do not coincide \cite{FTex532}. For clarity, we stress that $E_0^\bullet$ and $\hat{E}_0^\bullet$ are isomorphic subspaces of $\Omega^\bullet(N_{6,3,6})$, and this is necessarily true in general once we assume $M$ is a \textit{regular} subRiemannian manifold. However, they need not be the same subspace, and in this explicit example they are not (it is sufficient to compare $E_0^2$ and $\hat{E}_0^2$). In other words, just extending the metric $g_\cD$ from $\cD$ to an arbitrary Riemannian metric on $TM$ is not sufficient to ensure compatibility. However, compatible metrics can always be constructed.

\begin{lemma}
Let $M$ be a (second countable) regular subRiemannian manifold.
We can always construct a compatible Riemannian metric.
\end{lemma}

\begin{proof}
    Consider a sequence of non-negative functions $\varphi_i \in C_c^\infty (M)$, $i\in \bN$, such that 
    the sum $\sum_{i\in\mathbb{N}} \varphi_i$ on $M$ is locally finite and constantly equal to 1.
    We may assume that the support of each $\varphi_i$ is small enough so that it is included in an open subset $U_i\subset M$ where a frame $\bX_i$ of $TU_i$ exists. After a Graham-Schmidt procedure, 
    we may assume that each $\bX_i$ is such that the corresponding frame $\langle \bX_i\rangle$ of the vector bundle $\fg M|_{U_i}$ is $g_{\fg M}$-orthonormal.
Denoting by $g_i$ the corresponding metric on $U_i$ such that $\bX_i$
is orthonormal, one can easily check that 
$g=\sum_{i\in\mathbb N} \varphi_i g_i$ is a compatible Riemannian metric.   
\end{proof}

\subsubsection{The base differential and codifferential associated with a compatible Riemannian metric}
\label{subsubsec_baseM}

Let $M$ be a regular subRiemannian manifold equipped with a Riemannian metric $g_{TM}$.
 The compatibility implies that the scalar product on $\wedge^\bullet \gr(T_x^* M)$ (defined as in Section \ref{subsubsec_PhiM} using $g_{TM}$) 
 coincides with the scalar product on $\wedge^\bullet \fg_x^* M$ induced by the osculating metric $g_{\fg M}$.
Hence, the map $\Phi$ identifying forms and osculating forms and the objects built upon it (for instance 
 the base differential and co-differential $d_0$ and $\delta_0$) are defined globally as objects leaving on the manifolds. These removes  ambiguities and improves on earlier constructions 
 \cite{Rumin1990,Rumin1999,RuminPalermo}.

\appendix

\section{Case of a three-dimensional contact manifold (locally)}
\label{sec_contact}

On a three-dimensional contact manifold, the existence of Darboux coordinates implies that the manifold is filtered and can be locally described as follows. The manifold $M$ is equipped with a  frame $X,Y,T$ satisfying
\begin{align*}
	[X,Y] &= c_0 T +c_1 X+c_2 Y,\\
	[X,T] &= c_3 X + c_4 Y,\\
		[Y,T] &= c_5 X + c_6 Y,
\end{align*}
  with $c_j\in C^\infty(M)$, $j=0,\ldots,6$, with $c_0$ nowhere vanishing. 
  
The associated  free basis $\Theta^\bullet$ of $\Omega^\bullet( M)$ is given by
\begin{align*}
	k=0& \qquad 1(\geq 0),\\
	k=1&\qquad X^* (\geq 1) ,\ Y^* (\geq 1), \  T^*(\geq 2),\\
	k=2&\qquad X^* \wedge Y^*(\geq 2),\ X^* \wedge T^*(\geq 3), \ Y^* \wedge T^*(\geq 3),\\
	k=3 &\qquad\vol= X^* \wedge Y^*\wedge T^*(\geq 4).
\end{align*}
The numbers in parenthesis refer to the $\geq $ weight of the form.

\subsubsection*{Description of $\tilde d$}
For $k=0,3$, $\tilde d^{(k)}=0$. For $k=1,2$, let us describe $\tilde d$  
 in matrix form with respect to the canonical basis described above.
$$
\Mat (\tilde d^{(1)}) = \left(\begin{array}{ccc} 
	-c_1 & -c_2 & -c_0 \\ -c_3 & -c_4 &0 \\ -c_5 & -c_6 &0
\end{array}
\right) , 
\qquad
\Mat( \tilde d^{(2)}) = \left(\begin{array}{ccc} 
	c_3 + c_6 & -c_1 & -c_2\end{array}
\right) .
$$

\subsubsection*{Description of $d$}
For $k=3$, we have $d^{(3)}=0$, while 
for $k=0$, we have $d f (V) = Vf$, $f\in C^\infty (M)$.
Hence, 
\begin{align*}
\Mat(d^{(0)})&= 
\left(\begin{array}{c}
X\\Y\\T 	
\end{array}\right)\\
\Mat(d^{(1)})
&= 
\left(\begin{array}{ccc}
-Y &X& 0 \\ -T & 0& X \\ 0&- T &Y 	
\end{array}\right)
+\Mat(\tilde d^{(1)})
\\	
\Mat(d^{(2)})
&= 
\left(\begin{array}{ccc}
T & -Y & X
\end{array}\right)
+\Mat(\tilde d^{(2)})
\\
\end{align*}

\subsubsection*{Description of $d_{\fg M}$}
The osculating group is isomorphic to the Heisenberg group at any point $p\in M$:
$$
	[\langle X\rangle_p,\langle Y\rangle_p ]_{\fg_x M} = c_0(p) \langle T\rangle_p,
	\qquad
	[\langle X\rangle_p,\langle T\rangle_p ]_{\fg_x M} 
	=
	0
	=
	[\langle Y\rangle_p,\langle T\rangle_p ]_{\fg_x M} .
$$
The associated free basis of $G\Omega^\bullet (M)$ is 
\begin{align*}
	k=0& \qquad 1(= 0),\\
	k=1&\qquad \langle X^* \rangle (= 1) ,\ \langle Y^* \rangle(= 1), \ \langle T^*\rangle (= 2),\\
	k=2&\qquad \langle X^*\rangle \wedge \langle Y^*\rangle (= 2),\  \langle X^*\rangle \wedge \langle T^*\rangle (= 3), \ \langle Y^* \rangle \wedge \langle T^*\rangle (= 3),\\
	k=3 &\qquad \langle X^* \rangle \wedge \langle Y^*\rangle \wedge \langle T^*\rangle (= 4).
\end{align*}
The number in parenthesis refers to the weight of the osculating form. 

For $k=0,2,3$, $d_{\fg M}^{(k)}=0$. For $k=1$, let us describe $d_{\fg M}$  
 in matrix form with respect to the canonical basis described above.
$$
\Mat (d_{\fg M}^{(1)}) = \left(\begin{array}{ccc} 
	0 & 0 & -c_0 \\ 0 & 0 &0 \\ 0 & 0 &0
\end{array}
\right) . $$

\subsubsection*{Description of $d_0$ and $d_0^t$}
For $k=0,2,3$, $d_0^{(k)}=0$ while for $k=1$, 
the description of $d_0^{(1)}$  
 in matrix form with respect to the canonical basis of $\Omega^\bullet (M)$ given above coincides with  $\Mat (d_{\fg M}^{(1)})$.
For $k=0,1,3$, $d_0^{(k,t)}=0$. 
For $k=2$, the matrix description of $d_0^t$  coincides with the transpose of  $\Mat (d_{\fg M}^{(1)})$.
 Consequently, 
$$
\Mat (d_0^{(1)}) = \left(\begin{array}{ccc} 
	0 & 0 & -c_0 \\ 0 & 0 &0 \\ 0 & 0 &0
\end{array}
\right)\ , 
\quad
\Mat (d_0^{(1,t)}) = \left(\begin{array}{ccc} 
	0 & 0 & 0 \\ 0 & 0 &0 \\ -c_0 & 0 &0
\end{array}
\right) . $$

\subsubsection*{Description of $\Box_0$, $P_0$}
For $k=0,3$, $\Box^{(k)}_0=0$, while $\Box^{(1)}_0=d_0^{(1,t)}d_0^{(1)}$ and
$\Box^{(2)}_0=d_0^{(1)}d_0^{(1,t)}$ are represented by the following diagonal matrices:
\begin{align*}
	\Mat \big(\Box_0^{(1)} \big)&= \Mat \big(d_0^{(1,t)}\big)\Mat \big(d_0^{(1)}\big) 
	= \diag \big(0,0,c_0^2\big)\ ,
\\
\Mat \big(\Box_0^{(2)}\big)& = \Mat \big(d_0^{(1)}\big)\Mat \big(d_0^{(1,t)}\big)
= \diag \big(c_0^2,0,0\big)\ .
\end{align*}
Consequently, 
$\Pi_0^{(k)}$ is the identity for $k=0,3$, while for $k=1,2$, 
$$
\Mat \big(\Pi_0^{(1)}\big) = \diag (1,1,0)
\qquad\mbox{and}\qquad
\Mat \big(\Pi_0^{(2)}\big) = \diag (0,1,1)\ .
$$
Consequently, $E_0^{(0)}= \Omega^0(M) = \Phi^{-1} \big(\langle 1\rangle\big)$
and $E_0^{(3)}= \Omega^3(M) = \Phi^{-1}\big(\langle X^* \wedge Y^* \wedge T^*\rangle\big)$ while 
$$    
E_0^{(1)}= \Phi^{-1}\big(\langle X^*\rangle\oplus\langle Y^*\rangle\big)
\quad\mbox{and}\quad
E_0^{(2)}= \Phi^{-1}\big(\langle X^*\wedge T^*\rangle \oplus \langle Y^*\wedge T^*\rangle\big).
$$

\subsubsection*{Description of $\Box$, $P$, $L$}
For $k=0,3$, we have
$\Box^{(k)} = 0$, while for $k=1,2$
\begin{align*}
	\Mat \big(\Box^{(1)} \big)&= \Mat \big(d_0^{(1,t)}\big)\Mat \big(d^{(1)}\big) 
	= c_0 \left(\begin{array}{ccc} 
	0 & 0 & 0 \\ 0 & 0 &0 \\  Y +c_1 & - X  +c_2 & c_0
\end{array}
\right),
\\
\Mat \big(\Box^{(2)}\big)& = \Mat \big(d^{(1)}\big)\Mat \big(d_0^{(1,t)}\big)
= - c_0
\left(\begin{array}{ccc} 
	-c_0 & 0 & 0 \\  X & 0 &0 \\ Y  & 0  & 0
\end{array}
\right).
\end{align*}
Consequently, 
$P^{(k)}$ is the identity on $\Omega^k( M)$ for $k=0,3$. Since  $\big((z-\Box_0)(\Box_0 -\Box)\big)^j =0$ for $j>1$, the formulae for Cauchy residues  and the von Neumann series simplify into
$$
P = \Pi_0 + \Pi_0 (\Box_0 -\Box) c_0^{-2} \pr_{c_0^2} 
+c_0^{-2} \pr_{c_0^2}  (\Box_0 -\Box)\Pi_0 \
 ,$$
where $\pr_{c_0^2} $ denotes the projection onto the $c_0^2$-eigenspace of $\Box_0$.
Hence, for $k=1,2$, 
\begin{align*}
\Mat \big(P^{(1)}\big) &=  \diag (1,1,0) -c_0^{-1} \left(\begin{array}{ccc} 
	0 & 0 & 0 \\ 0 & 0 &0 \\ Y +c_1 & - X  +c_2 & 0
\end{array}
\right)\ ,\\	
\Mat \big(P^{(2)}\big) &=  \diag (0,1,1) +c_0^{-1}\left(\begin{array}{ccc} 
	0 & 0 & 0 \\  X & 0 &0 \\  Y  & 0  & 0
\end{array}
\right)\ .
\end{align*}

We now describe 
$$
L=P\Pi_0 +(\id-P)(\id-\Pi_0)
= \id +(P-\Pi_0)(-\id+2\Pi_0). 
$$
For $k=0,3$, $L^{(k)}$ is the identity on $\Omega^k (M)$
while for $k=1,2$, 
\begin{align*}
\Mat \big(L^{(1)}\big) &=  \id_3 -c_0^{-1} \left(\begin{array}{ccc} 
	0 & 0 & 0 \\ 0 & 0 &0 \\ Y +c_1 & - X  +c_2 & 0
\end{array}
\right)
\\	
\Mat \big(L^{(2)}\big) &=  \id_3 -c_0^{-1}\left(\begin{array}{ccc} 
	0 & 0 & 0 \\  X & 0 &0 \\  Y  & 0  & 0
\end{array}
\right),
\end{align*}
where $\id_3=\diag(1,1,1)$ denotes the 3-by-3 identity matrix. 

\subsubsection*{Computation of $D$}
We now describe 
$$
D^{(k)}=\big(L^{(k+1)}\big)^{-1}d^{(k)} L^{(k)}\Pi_0^{(k)},
$$ 
for $k=0,1,2$.
\begin{align*}
\Mat\big(D^{(0)}\big)&=	\left(\begin{array}{c}
X\\Y\\T +c_0^{-1} ( (Y+c_1) X+ (-X+c_2)Y)	
\end{array}\right)
=	\left(\begin{array}{c}
X\\Y\\0	
\end{array}\right)
,\\
\Mat(D^{(1)})&=	\left(\begin{array}{ccc}
0&0&0\\
-T-c_3 -c_0^{-1}X (Y+c_1)&-c_4 +c_0^{-1}X (X-c_2) &0 \\
-c_5-c_0^{-1}Y(Y+c_1)  &-T-c_6 +Yc_0^{-1}(X-c_2)  &0
\end{array}\right),\\
\Mat(D^{(2)})&=	\left(\begin{array}{ccc}
0&-Y-c_1&X-c_2
\end{array}\right).
\end{align*}
The zeros in the matrices above illustrate $D$ acting on $E_0$ while being trivial on $F_0$.

\subsubsection*{Description of $\tilde\Box$, $\tilde P$, $\tilde L$} For $k=0,3$, we have $\tilde\Box^{(k)}=0$, while for $k=1,2$
\begin{align*}
\Mat \left(\tilde\Box^{(1)}\right)=&\Mat\big( d_0^{(1,t)}\big)\Mat\big(\tilde d^{(1)}\big)=c_0\left(\begin{array}{ccc}
     0&0&0  \\
     0&0&0\\ c_1&c_2&c_0 
\end{array}\right)\ ,\\
\Mat\big(\tilde\Box^{(2)}\big)=&\Mat\big(\tilde d^{(1)}\big)\Mat\big(d_0^{(1,t)}\big)=\diag\big(c_0^2,0,0\big)\ .
\end{align*}
Consequently, $\tilde P^{(k)}$ is the identity on $\Omega^k(M)$ for $k=0,3$. Since $\big((z-\Box_0)(\Box_0-\tilde\Box)\big)^j=0$ for $j>1$, the formulae for Cauchy residues and the von Neumann series simplify into
\begin{align*}
    \tilde P=\Pi_0+\Pi_0(\Box_0-\tilde\Box)c_0^{-2}\pr_{c_0^2}+c_0^{-2}\pr_{c_0^2}(\Box_0-\tilde\Box)\Pi_0\ ,
\end{align*}
where $\pr_{c_0^2}$ denotes the projection onto the $c_0^2$-eigenspace of $\Box_0$. Hence, for $k=1,2$,
\begin{align*}
    \Mat\big(\tilde P^{(1)}\big)=&\diag(1,1,0)-c_0^{-1}\left(\begin{array}{ccc}
        0&0&0\\ 0&0&0\\      
        c_1&c_2&0 
\end{array}\right)\\\Mat\big(\tilde P^{(2)}\big)=&\diag(0,1,1)\ .
\end{align*}
We now describe
\begin{align*}
    \tilde L=\tilde P\Pi_0+(\id-\tilde P)(\id-\Pi_0)=\id+(\tilde P-\Pi_0)(-\id+2\Pi_0)\ .
\end{align*}
For $k=0,3$, $\tilde L^{(k)}$ is the identity on $\Omega^k(M)$ while for $k=1,2$,
\begin{align*}
    \Mat\big(\tilde L^{(1)}\big)=\id_3-c_0^{-1}\left(\begin{array}{ccc}
    0&0&0\\0&0&0\\
    c_1& c_2&0 
\end{array}\right) \ ,\ \Mat\big(\tilde L^{(2)}\big)=\id_3\ ,
\end{align*}
where $\id_3=\diag(1,1,1)$ denotes the 3-by-3 identity matrix.

\subsubsection*{Computation of $\tilde D$} We now describe
\begin{align*}
    \tilde D^{(k)}=\big(L^{(k+1)}\big)^{-1}\tilde d^{(k)}\tilde L^{(k)}\Pi_0^{(k)}
\end{align*}
for $k=0,1,2$.
$$
    \Mat\big(\tilde D^{(0)}\big)=\left(\begin{array}{c}
        0\\0\\0   
\end{array}\right),\quad
\Mat\big(\tilde D^{(1)}\big)=\left(\begin{array}{ccc}
0&0&0\\ -c_3&-c_4&0\\
     -c_5&-c_6&0 
\end{array}\right),
$$
$$
\Mat\big(\tilde D^{(2)}\big)=\left(\begin{array}{ccc}0 &-c_1&-c_2
\end{array} \right).
$$

\section{Examples}
\label{sec_app_ex}

In this appendix, 
we give two examples of a differential operator $D\colon\Omega^\bullet(M)\to\Omega^\bullet(M)$ that respects the filtration \eqref{filtration on k forms by weight} whose transpose  $D^t$ does not necessarily respect the filtration. 

\begin{ex}
\label{ex1_DtnotrespectingF}
We consider the 3-dimensional Heisenberg group $\bH$, and the canonical basis $X,Y,T$ of its Lie algebra $\fh$.
The only non-trivial bracket is $[X,Y]=T$.
Identifying $\fh$ with the space of left-invariant vector fields on $\bH$, we obtain a filtration of the tangent bundle of $\bH$ in subbundles:
$$
 \lbrace 0\rbrace=H^0\subset H^1:={\rm span}_{\bR}\lbrace X,Y\rbrace\subset H^2:={\rm span}_{\bR}\lbrace X,Y,T\rbrace=T\bH^1\,.
$$
The dual of the  left-invariant frame $\bX:=(X,Y,T)$ of $T\bH$ 
yields the global coframe $\Theta:=(X^\ast,Y^\ast,T^\ast)$
(as well as the natural linear isomorphism between  $\Omega^\bullet (\bH)$ and 
$C^\infty (\bH)\otimes \wedge ^\bullet \fh^*$).

Let us now consider the de Rham differential $d\colon\Omega^\bullet(\bH)\to\Omega^\bullet(\bH)$ acting on the 1-form $f\, T^\ast\in\Omega^{1,\ge 2}(\bH)$ with $f\in C^\infty(\bH)$:
    \begin{align*}
        d(fT^\ast)=df\wedge T^\ast+f\,  dT^\ast=Xf\, (X^\ast\wedge T^\ast)+Yf\, (Y^\ast\wedge T^\ast)-f\, (X^\ast\wedge Y^\ast)\in\Omega^{2,\ge 2}(\bH)\,.
    \end{align*}
    If we now consider its formal transpose $d^t=(-1)^{k\cdot 3+1}\star d\star\colon\Omega^{k+1}(\bH)\to\Omega^k(\bH)$ acting on the 2-form $g\,  (X^\ast\wedge Y^\ast)\ \in\ \Omega^{2,\ge 2}(\bH)$ with $g\in C^\infty(\bH)$:
    \begin{align*}
        d^t(g\, (X^\ast\wedge Y^\ast))=&\star d (g\, T^\ast)=\star(-g\, (X^\ast\wedge Y^\ast)+Xg\,(X^\ast\wedge T^\ast)+Yg\,(Y^\ast\wedge T^\ast))\\=&-g\, T^\ast-Xg\,Y^\ast+Yg\,X^\ast\in\Omega^{1,\ge 1}(\bH)
    \end{align*}
    and so $d^t(\Omega^{2,\ge 2}(\bH))\subset \Omega^{1,\ge 1}(\bH)$, but not $\Omega^{1,\ge 2}(\bH)$. 

    Applying similar computations, one can easily check that the transpose of the operator $\gr(d)^\Phi$ instead respects the filtration.

\end{ex}

A  situation similar to Example \ref{ex1_DtnotrespectingF} also holds for the algebraic part of the de Rham differential $\tilde d\colon\Omega^\bullet(\bG)\to\Omega^\bullet(\bG)$, when considering a nilpotent Lie group $\bG$ with a filtration on its Lie algebra $\fg$ that is not coming from a homogeneous structure as in the following example:

\begin{ex}
\label{ex2_DtnotrespectingF}
Let us consider the 4-dimensional Engel group $\bG$, that is, the connected simply connected nilpotent Lie group with Lie algebra $\fg={\rm span}_\bR\lbrace X_1,X_2,X_3,X_4\rbrace$ and $[X_1,X_i]=X_{i+1}$ for $i=2,3$.
We identify $\fg$ with the space of left-invariant vector fields, and we consider the following left-invariant filtration of $T\bG$
    \begin{align*}
        \lbrace 0\rbrace=H^0\subset H^1={\rm span}_\bR\lbrace X_1,X_2\rbrace\subset H^2={\rm span}_\bR\lbrace X_1,X_2,X_3,X_4\rbrace=T\bG\,.
    \end{align*}
    The coframe of the canonical frame $\bX=(X_1,,X_2,X_3,X_4)$ is 
    denoted by $\Theta=(\theta^1,\theta^2,\theta^3,\theta^4)$
    Then, for an arbitrary form $f_1\, \theta^3+f_2\, \theta^4$ in $\Omega^{1,\ge 2}(\bG)$, with  $f_1,f_2\in C^\infty(\bG)$, we have
    \begin{align*}
        \tilde d(f_1\, \theta^3+f_2\, \theta^4)=&-f_1\, (\theta^1\wedge\theta^2)-f_2\, (\theta^1\wedge\theta^3) \ \in \Omega^{2,\ge 2}(\bG)\, . 
    \end{align*}
    However, given a form $g\, (\theta^1\wedge\theta^3) \ \in \Omega^{2,\ge 3}(\bG)$  with $g \in C^\infty(\bG)$, we obtain $\tilde d^t(g\, (\theta^1\wedge\theta^3))=-g\, \theta^4\in\Omega^{1,\ge 2}(\bG)$.
\end{ex}

\bibliographystyle{alpha.bst}

\bibliography{bibli}

\begin{thebibliography}{GAJV19}

\bibitem[BFP22]{BaldiFranchiPansu22}
Annalisa Baldi, Bruno Franchi, and Pierre Pansu.
\newblock Poincar\'{e} and {S}obolev inequalities for differential forms in
  {H}eisenberg groups and contact manifolds.
\newblock {\em J. Inst. Math. Jussieu}, 21(3):869--920, 2022.

\bibitem[BvE14]{BaumvErp}
Paul~F. Baum and Erik van Erp.
\newblock {$K$}-homology and index theory on contact manifolds.
\newblock {\em Acta Math.}, 213(1):1--48, 2014.

\bibitem[Can21]{canarecci2021sub}
Giovanni Canarecci.
\newblock Sub-{R}iemannian currents and slicing of currents in the {H}eisenberg
  group {$\Bbb{H}^n$}.
\newblock {\em J. Geom. Anal.}, 31(5):5166--5200, 2021.

\bibitem[Cas11]{case2022}
Jeffrey~S. Case.
\newblock The bigraded rumin complex via differential forms, arXiv:2108.13911.

\bibitem[CCY16]{CaseChanilloYang}
Jeffrey~S. Case, Sagun Chanillo, and Paul Yang.
\newblock The {CR} {P}aneitz operator and the stability of {CR} pluriharmonic
  functions.
\newblock {\em Adv. Math.}, 287:109--122, 2016.

\bibitem[CD01]{CalderbankDiemer}
David M.~J. Calderbank and Tammo Diemer.
\newblock Differential invariants and curved {B}ernstein-{G}elfand-{G}elfand
  sequences.
\newblock {\em J. Reine Angew. Math.}, 537:67--103, 2001.

\bibitem[{\v{C}}H23]{CapHu}
Andreas {\v{C}}ap and Kaibo Hu.
\newblock Bounded {P}oincar\'{e} operators for twisted and {BGG} complexes.
\newblock {\em J. Math. Pures Appl. (9)}, 179:253--276, 2023.

\bibitem[{\v{C}}SS01]{CapSlovakSouvek}
Andreas {\v{C}}ap, Jan Slov\'{a}k, and Vladim\'{\i}r Sou\v{c}ek.
\newblock Bernstein-{G}elfand-{G}elfand sequences.
\newblock {\em Ann. of Math. (2)}, 154(1):97--113, 2001.

\bibitem[DH22]{DaveHaller22}
Shantanu Dave and Stefan Haller.
\newblock Graded hypoellipticity of {BGG} sequences.
\newblock {\em Ann. Global Anal. Geom.}, 62(4):721--789, 2022.

\bibitem[FT12]{franchitesi}
Bruno Franchi and Maria~Carla Tesi.
\newblock Wave and {M}axwell's equations in {C}arnot groups.
\newblock {\em Commun. Contemp. Math.}, 14(5):1250032, 62, 2012.

\bibitem[FT23]{F+T1}
V\'{e}ronique Fischer and Francesca Tripaldi.
\newblock An alternative construction of the {R}umin complex on homogeneous
  nilpotent {L}ie groups.
\newblock {\em Adv. Math.}, 429:Paper No. 109192, 39, 2023.

\bibitem[FT24]{FTex532}
V{\'e}ronique Fischer and Francesca Tripaldi.
\newblock Extracting subcomplexes on nilpotent groups without relying on
  homogeneous structures, in preparation, 2024.

\bibitem[GAJV19]{JulgSurvey}
Maria~Paula Gomez~Aparicio, Pierre Julg, and Alain Valette.
\newblock The {B}aum-{C}onnes conjecture: an extended survey.
\newblock In {\em Advances in noncommutative geometry---on the occasion of
  {A}lain {C}onnes' 70th birthday}, pages 127--244. Springer, Cham, [2019]
  \copyright 2019.

\bibitem[GT53]{grong2023filtered}
Erlend Grong and Francesca Tripaldi.
\newblock Filtered complexes and cohomologically equivalent subcomplexes,
  arXiv:2308.11353.

\bibitem[Hal22]{Haller22}
Stefan Haller.
\newblock Analytic torsion of generic rank two distributions in dimension five.
\newblock {\em J. Geom. Anal.}, 32(10):Paper No. 248, 66, 2022.

\bibitem[JK95]{JulgKasparov}
Pierre Julg and Gennadi Kasparov.
\newblock Operator {$K$}-theory for the group {${\rm SU}(n,1)$}.
\newblock {\em J. Reine Angew. Math.}, 463:99--152, 1995.

\bibitem[JP12]{julia2023}
Antoine Julia and Pierre Pansu.
\newblock Flat compactness of normal currents, and charges in carnot groups,
  arXiv:2303.02012.

\bibitem[Jul95]{JulgSpectral}
Pierre Julg.
\newblock Complexe de {R}umin, suite spectrale de {F}orman et cohomologie
  {$L^2$} des espaces sym\'etriques de rang {$1$}.
\newblock {\em C. R. Acad. Sci. Paris S\'er. I Math.}, 320(4), 1995.

\bibitem[Jul19]{Julg2019}
Pierre Julg.
\newblock How to prove the {B}aum-{C}onnes conjecture for the groups
  {$Sp(n,1)$}?
\newblock {\em J. Geom. Phys.}, 141:105--119, 2019.

\bibitem[Kit20]{kitaoka2020analytic}
Akira Kitaoka.
\newblock Analytic torsions associated with the {R}umin complex on contact
  spheres.
\newblock {\em Internat. J. Math.}, 31(13):2050112, 16, 2020.

\bibitem[KMX28]{kleiner2021sobolev}
Bruce Kleiner, Stefan Muller, and Xiangdong Xie.
\newblock Sobolev mappings and the rumin complex, arXiv:2101.04528.

\bibitem[Kos61]{Kostant}
Bertram Kostant.
\newblock Lie algebra cohomology and the generalized {B}orel-{W}eil theorem.
\newblock {\em Ann. of Math. (2)}, 74:329--387, 1961.

\bibitem[LDT22]{LD+T22}
Enrico Le~Donne and Francesca Tripaldi.
\newblock A cornucopia of {C}arnot groups in low dimensions.
\newblock {\em Anal. Geom. Metr. Spaces}, 10(1):155--289, 2022.

\bibitem[LT23]{LerarioTripaldi}
Antonio Lerario and Francesca Tripaldi.
\newblock Multicomplexes on {C}arnot groups and their associated spectral
  sequence.
\newblock {\em J. Geom. Anal.}, 33(7):Paper No. 199, 22, 2023.

\bibitem[Mit85]{Mitchell}
John Mitchell.
\newblock On {C}arnot-{C}arath\'{e}odory metrics.
\newblock {\em J. Differential Geom.}, 21(1):35--45, 1985.

\bibitem[Mon02]{Montgomerybk}
Richard Montgomery.
\newblock {\em A tour of subriemannian geometries, their geodesics and
  applications}, volume~91 of {\em Mathematical Surveys and Monographs}.
\newblock American Mathematical Society, Providence, RI, 2002.

\bibitem[PR18]{PansuRumin}
Pierre Pansu and Michel Rumin.
\newblock On the {$\ell^{q,p}$} cohomology of {C}arnot groups.
\newblock {\em Ann. H. Lebesgue}, 1:267--295, 2018.

\bibitem[PT19]{pansu2019averages}
Pierre Pansu and Francesca Tripaldi.
\newblock Averages and the {$\ell^{q,1}$} cohomology of {H}eisenberg groups.
\newblock {\em Ann. Math. Blaise Pascal}, 26(1):81--100, 2019.

\bibitem[RS76]{RotschildStein}
Linda~Preiss Rothschild and E.~M. Stein.
\newblock Hypoelliptic differential operators and nilpotent groups.
\newblock {\em Acta Math.}, 137(3-4):247--320, 1976.

\bibitem[RS12]{RuminSeshadri}
Michel Rumin and Neil Seshadri.
\newblock Analytic torsions on contact manifolds.
\newblock {\em Ann. Inst. Fourier (Grenoble)}, 62(2):727--782, 2012.

\bibitem[Rum90]{Rumin1990}
Michel Rumin.
\newblock Un complexe de formes diff\'{e}rentielles sur les vari\'{e}t\'{e}s de
  contact.
\newblock {\em C. R. Acad. Sci. Paris S\'{e}r. I Math.}, 310(6):401--404, 1990.

\bibitem[Rum99]{Rumin1999}
Michel Rumin.
\newblock Differential geometry on {C}-{C} spaces and application to the
  {N}ovikov-{S}hubin numbers of nilpotent {L}ie groups.
\newblock {\em C. R. Acad. Sci. Paris S\'{e}r. I Math.}, 329(11):985--990,
  1999.

\bibitem[Rum05]{RuminPalermo}
Michel Rumin.
\newblock An introduction to spectral and differential geometry in
  {C}arnot-{C}arath\'eodory spaces.
\newblock {\em Rend. Circ. Mat. Palermo (2) Suppl.}, (75):139--196, 2005.

\bibitem[Tri54]{tripaldi2020rumin}
Francesca Tripaldi.
\newblock The {R}umin complex on nilpotent lie groups, arXiv:2009.10154.

\bibitem[vE10]{vanErp10}
Erik van Erp.
\newblock The {A}tiyah-{S}inger index formula for subelliptic operators on
  contact manifolds. {P}arts {I} and {II}.
\newblock {\em Ann. of Math. (2)}, 171(3):1683--1706 and 1647--1681, 2010.

\bibitem[vEY17]{vErpYuncken}
Erik van Erp and Robert Yuncken.
\newblock On the tangent groupoid of a filtered manifold.
\newblock {\em Bull. Lond. Math. Soc.}, 49(6):1000--1012, 2017.

\bibitem[Vit22]{Vittone}
Davide Vittone.
\newblock Lipschitz graphs and currents in {H}eisenberg groups.
\newblock {\em Forum Math. Sigma}, 10:Paper No. e6, 104, 2022.

\end{thebibliography}

\end{document}